\documentclass[12pt]{amsart}

\usepackage{graphicx}
\usepackage{amsmath}
\usepackage{amscd}
\usepackage{amsfonts}
\usepackage{amssymb}
\usepackage[dvipsnames]{xcolor}
\usepackage{fullpage}
\usepackage{mathtools}
\usepackage{graphicx}
\usepackage{caption}
\usepackage{subcaption}
\usepackage{verbatim}
\usepackage[pagebackref,hypertexnames=false, colorlinks, citecolor=red,linkcolor=blue, urlcolor=red]{hyperref}

\newcommand\blue{\color{blue}}

\numberwithin{equation}{section}

\theoremstyle{definition}
\newtheorem{theorem}[equation]{Theorem}
\newtheorem{lemma}[equation]{Lemma}
\newtheorem{proposition}[equation]{Proposition}

\newtheorem{corollary}[equation]{Corollary}
\newtheorem{problem}[equation]{Problem}
\newtheorem{definition}[equation]{Definition}
\newtheorem{example}[equation]{Example}
\newtheorem{remark}[equation]{Remark}

\newcommand{\be}{\begin{equation}}
\newcommand{\ee}{\end{equation}}
\newcommand{\bes}{\begin{equation*}}
\newcommand{\ees}{\end{equation*}}
\newcommand{\ba}{\begin{aligned}}
\newcommand{\ea}{\end{aligned}}

\newcommand{\cH}{H}
\newcommand{\cK}{\mathcal{K}}

\newcommand{\cB}{\mathcal{B}}

\newcommand{\cW}{\mathcal{W}}

\newcommand{\Md}{\mathcal{M}^d}

\newcommand{\bC}{\mathbb{C}}
\newcommand{\bD}{\mathbb{D}}

\newcommand{\bR}{\mathbb{R}}
\newcommand{\bS}{\mathbb{S}}

\newcommand{\bZ}{\mathbb{Z}}

\newcommand{\Wmin}[1]{\cW^{\text{min}}_{#1}}
\newcommand{\Wmax}[1]{\cW^{\text{max}}_{#1}}

\newcommand{\Diam}{\mbox{\larger\larger\larger$\diamond$}}
\newcommand{\USD}[1]{\mathcal{U}(#1)}

\newcommand\queseq{\stackrel{\mathclap{\normalfont\mbox{?}}}{=}}

\newcommand{\eucball}[1]{\overline{\mathbb{B}^{#1}}}

\newcommand{\ol}{\overline}

\begin{document}

\title{Shape, Scale, and Minimality of Matrix Ranges}
\date{\today}

\author[Passer]{$\text{Benjamin Passer}^{\dag}$}
\thanks{$\dag$ Partially supported by a Zuckerman Fellowship at the Technion.}
\address{Faculty of Mathematics\\
Technion-Israel Institute of Technology\\
Haifa\\
Israel}
\email{benjaminpas@technion.ac.il}

\begin{abstract}
We study containment and uniqueness problems concerning matrix convex sets. First, to what extent is a matrix convex set determined by its first level? Our results in this direction quantify the disparity between two product operations, namely the product of the smallest matrix convex sets over $K_i \subseteq \mathbb{C}^d$, and the smallest matrix convex set over the product of $K_i$. Second, if a matrix convex set is given as the matrix range of an operator tuple $T$, when is $T$ determined uniquely? We provide counterexamples to results in the literature, showing that a compact tuple meeting a minimality condition need not be determined uniquely, even if its matrix range is a particularly friendly set. Finally, our results may be used to improve dilation scales, such as the norm bound on the dilation of (non self-adjoint) contractions to commuting normal operators, both concretely and abstractly.
\end{abstract}

\maketitle

\section{Introduction}

The noncommutative generalization of a function system is called an operator system, and many crucial objects in the study of function systems also generalize to the noncommutative setting \cite{arvesonhyp, hamana}.

\begin{definition}
An \textit{operator system} is a self-adjoint, unital subspace $S$ of a unital $C^*$-algebra $\mathcal{A}$.  If $\mathcal{A} = C(X)$ is commutative, then $S$ is also called a \textit{function system} on $X$.
\end{definition}

Any unital $C^*$-algebra is spanned by its positive (or more precisely, positive semidefinite) elements, which by definition are self-adjoint elements $a \in \mathcal{A}$ whose spectra are contained in the nonnegative real line. We write $a \geq 0$ when $a$ is positive, noting that positivity of $a$ is equivalent to the claim that a factorization $a = bb^*$ exists for some $b \in \mathcal{A}$. While an operator system $S \subseteq \mathcal{A}$ might not have a multiplicative structure of its own, by considering the given $C^*$-algebra in which $S$ lives, one may point out the set of positive elements in $S$. In particular, $S$ is also spanned by its positive elements, which make up a crucial part of the operator system structure, as in the abstract definition found in \cite{choieffros}.

The above discussion of positivity applies equally well to the set of $n \times n$ matrices over $S$, as $M_n(S)$ embeds into the unital $C^*$-algebra $M_n(\mathcal{A})$. Further, to any map $\phi: S \to T$ between operator systems, one may also produce maps $\phi^{(n)}: M_n(S) \to M_n(T)$ which apply $\phi$ entrywise. The relevant notion of morphism between operator systems is a map which respects all of the above structure, on every matrix level, as in the following definition.

\begin{definition}
Let $S \subseteq \mathcal{A}$ and $T \subseteq \mathcal{B}$ be operator systems. Then a linear map $\phi: S \to T$ is a \textit{unital completely positive} map, or \textit{UCP} map, if $\phi(1) = 1$ and each $\phi^{(n)}$ is positive -- for any matrix $s \in M_n(S)$ such that $s \geq 0$, it follows that $\phi^{(n)}(s) \geq 0$.
\end{definition}

Any unital $C^*$-algebra $\mathcal{A}$ is by default an operator system, and Arveson's extension theorem \cite[Theorem 1.2.3]{arvesonsubI} implies that for an operator system $S \subseteq \mathcal{A}$, any UCP map $S \to \cB(H)$ extends to a UCP map $\mathcal{A} \to \cB(H)$. Therefore, when studying the interpolation problem for UCP maps into $\cB(H)$, the choice of domain is generally not important. (In contrast, it is often of great interest if an extension of a UCP map given by Arveson's extension theorem is \textit{unique}, and this problem is certainly domain-sensitive. See, for example, the role of unique extensions in Arveson's Hyperrigidity Conjecture \cite[Conjecture 4.3]{arvesonhyp}.) The interpolation problem for UCP maps reduces to consideration of the matrix range, defined below.

\begin{definition}
The matrix range of $A = (A_1, \ldots, A_d) \in \cB(H)^d$, denoted $\mathcal{W}(A) = \bigcup\limits_{n=1}^\infty \mathcal{W}_n(A)$, is a subset of the matrix universe $\Md = \bigcup\limits_{n=1}^\infty M_n^d$ defined on each $n \times n$ level by
\bes
\mathcal{W}_n(A) := \{(Q_1, \ldots, Q_d) \in M_n^d: \exists \textrm{ UCP map } \psi: \cB(H) \to M_n \textrm{ with } \psi(A_i) = Q_i\}.
\ees
\end{definition}

From a slight reworking of \cite[Theorem 2.4.2]{arvesonsubII} in \cite[Theorem 5.1]{DDSS}, if $A \in \cB(H)^d$ and $B \in \cB(K)^d$, then a UCP map $\phi: \cB(H) \to \cB(K)$ mapping $\phi(A_i) = B_i$ exists precisely if $\mathcal{W}(B) \subseteq \mathcal{W}(A)$. Thus, the interpolation problem for UCP maps reduces to the consideration of (all) UCP maps whose codomains are finite-dimensional. For any $A$, the matrix range $\mathcal{W}(A)$ is a closed and bounded matrix convex set \cite[Proposition 2.5]{DDSS}. More precisely, each $\mathcal{W}_n(A)$ is closed as a subset of $M_n^d$ with the product norm topology, and there is a uniform bound (independent of $n$) on the norm of any member of any tuple belonging to $\mathcal{W}_n(A)$. Matrix convexity is defined as follows.

\begin{definition}
A set $\mathcal{S} = \bigcup\limits_{n=1}^\infty \mathcal{S}_n \subseteq \Md$ is \textit{matrix convex} if it is closed under the application of direct sums and UCP maps. That is, $\mathcal{S}$ meets the following conditions.
\begin{itemize}
\item If $X \in \mathcal{S}_n$ and $Y \in \mathcal{S}_m$, then $X \oplus Y \in \mathcal{S}_{n+m}$.
\item If $X \in \mathcal{S}_n$ and $\phi: M_n \to M_m$ is a UCP map, then $(\phi(X_1), \ldots, \phi(X_d)) \in \mathcal{S}_m$.
\end{itemize}
\end{definition}

In fact, operator systems and matrix convex sets are dual to each other \cite{effroswinkler}. From Choi's theorem \cite{choiUCP}, which characterizes completely positive maps between matrix algebras, an equivalent definition of matrix convexity follows.

\begin{definition}
Let $\mathcal{S}$ be a subset of $\Md$. For each $i \in \{1, \ldots, N\}$, let $Y^i \in \mathcal{S}_{n_i}$ and let $V_i: \bC^n \to \bC^{n_i}$ be a linear map, such that $\sum\limits_{i=1}^N V_i^* V_i = I_n$. Then 
\be\label{eq:mccdef}
X := \sum_{i=1}^N V_i^* Y^i V_i
\ee
is called a \textit{matrix convex combination} of $Y^1, \ldots, Y^N \in \mathcal{S}$.
\end{definition}

\begin{definition}
A set $\mathcal{S} \subseteq \Md$ is \textit{matrix convex} if whenever $X$ is a matrix convex combination of $Y^1, \ldots, Y^N \in \mathcal{S}$, it follows that $X \in \mathcal{S}$.
\end{definition}

Note in particular that if $X \in \mathcal{S}_n$ and $Y \in \mathcal{S}_m$, then $X \oplus Y = V_1^* X V_1 + V_2^* Y V_2$ for a natural choice of coisometries $V_1$ and $V_2$, so $X \oplus Y$ is a matrix convex combination of $X$ and $Y$. Further, any unitary conjugation $U^* X U$ is a matrix convex combination of $X$ that uses only one summand.

It is of great interest what the proper notion of extreme point should be in the matrix convex setting, analogous to the Krein-Milman theorem (as well as Milman's converse) and the Minkowski/Steinitz theorem in the compact convex setting \cite{kreinmilman, steinitz}. Two major candidates are \textit{matrix extreme points} and \textit{absolute extreme points}, and both have been characterized in dilation-theoretic terms \cite[Theorem 1.1]{evertheltonext}. However, both candidates have limitations. If $\mathcal{S}$ is a closed and bounded matrix convex set, then $\mathcal{S}$ is generated by its matrix extreme points, in the sense that the smallest closed matrix convex set containing these points is equal to $\mathcal{S}$. However, it is possible for a matrix extreme point to be a nontrivial matrix convex combination of other matrix extreme points, and it is not known if there is a smaller generating set. While the definition of an absolute extreme point forces its representation as a matrix convex combination to be essentially unique, it is possible for $\mathcal{S}$ to have no absolute extreme points at all \cite[Corollary 1.1]{evertabs}. 

Given a compact convex set $K \subseteq \bC^d$, there might be many matrix convex sets $\mathcal{S}$ such that $\mathcal{S}_1 = K$, but there is always a smallest and largest choice of $\mathcal{S}$ \cite[Definition 4.1 and Proposition 4.3]{DDSS}. They may be presented in multiple equivalent ways:
\be\label{eq:wmindef} \ba
\Wmin{}(K) 	&= \{X \in \Md: \textrm{there is a normal dilation } N \textrm{ of } X \textrm{ with } \sigma(N) \subseteq K\}\\
		&= \{X \in \Md: \textrm{there is a normal matrix dilation } N \textrm{ of } X \textrm{ with } \sigma(N) \subseteq K\} \\
\ea \ee
and
\be\label{eq:wmaxdef} \ba
\Wmax{}(K)	&= \left\{X  \in \Md: \text{ if } \textrm{Re}\left( \sum\limits_{j=1}^d a_j x_j \right) \leq b \textrm{ for all } x \in K, \textrm{ then } \textrm{Re}\left( \sum\limits_{j=1}^d a_j X_j \right) \leq bI \right\} \\
		&=  \{X  \in \Md: \mathcal{W}_1(X) \subseteq K \}.
\ea \ee
We remind the reader that when a tuple $N = (N_1, \ldots, N_d)$ is called \textit{normal}, this means that the operators $N_1, \ldots, N_d$ are normal and commute with each other. Further, $Y \in \cB(K)^d$ is a \textit{dilation} of $X \in \cB(H)^d$ if there exists an isometry $V: H \to K$ such that $X_i = V^* Y_i V$ for each $i$. Equivalently, $X$ is a \textit{compression} of $Y$.

There is a considerable amount of information buried in the previous definitions. For example, as $\Wmin{}(K)$ is by definition the smallest matrix convex set spanned by the scalar set $K$, this spanning property does not include a closure operation. However, the second formulation of $\Wmin{}(K)$ (with added dimension bounds) shows it is actually closed. More detail from \cite{DDSS} and \cite{PSS} is given below, and the following proposition may be seen as a manipulation of the Stinespring dilation procedure \cite{stinespring}.
\begin{proposition}\label{prop:minWrange}(\cite[Corollary 2.8 and Proposition 4.4]{DDSS})
If $N \in \cB(H)^d$ is a normal tuple, then $\mathcal{W}(N) = \Wmin{}(K)$, where $K$ is the convex hull of $\sigma(N)$.
\end{proposition}

Since every compact convex set $K \subseteq \bC^d$ may be written as the convex hull of $\sigma(N)$ for a diagonal operator $N$, it follows that any $\Wmin{}(K)$ may be written as a matrix range $\mathcal{W}(N)$, and it is therefore also closed and bounded. Alternatively, a compactness argument can be used, as if $T \in \Wmin{n}(K)$, then there is a normal tuple of \textit{matrices} with a fixed dimension bound (depending on $n$) which dilates $T$ and has spectrum in $K$. 

\begin{proposition}\label{prop:introcara}(reformulation of \cite[Theorem 7.1]{DDSS} and \cite[Proposition 2.3]{PSS})
 If $T \in \cB(H)^d$ has $\mathcal{W}(T) \subseteq \Wmin{}(K)$, then there is a normal dilation $N$ of $T$ with $\sigma(N) \subseteq \overline{\text{ext}(K)}$. If, in addition, $T$ acts on a finite-dimensional space of dimension $n$, then we may choose $N$ to act on a space of dimension $2n^3(d+1) +1$ or lower.
\end{proposition}

Note that since $T$ above need not act on a finite-dimensional space, Proposition \ref{prop:introcara} shows that the existence of normal dilations for a family of matrices may be used to produce a normal dilation of an infinite-dimensional operator. In the background of this claim lies the fact that UCP maps (and hence Stinespring dilations) behave very well with respect to limits in pointwise topologies.

For most of the problems we pursue, there is no harm in considering a tuple $(T_1, \ldots, T_d)$ of $d$ operators as a tuple $(X_1, Y_1, \ldots, X_d, Y_d)$ of $2d$ self-adjoint operators instead. In particular, this notational change does not affect the definitions of $\mathcal{W}(T)$, $\Wmin{}(K)$, and $\Wmax{}(K)$. Therefore, whenever it is possible, we will restrict proofs to the self-adjoint setting $\cB(H)^d_{sa}$ and assume that any scalar tuples we consider belong to $\bR^d$. We also note that it is possible for $\Wmax{}(K)$ and $\Wmin{}(K)$ to be equal, and this occurs if and only if $K$ is a simplex. More specifically, for a compact convex set $K \subseteq \bR^d$, 
\be\label{eq:simplex2tod-1}
\Wmax{}(K) = \Wmin{}(K) \iff \Wmax{2^{d-1}}(K) = \Wmin{2^{d-1}}(K) \iff K \textrm{ is a simplex }
\ee
holds from \cite[Theorem 4.1]{PSS}. See also \cite[Theorem 4.7]{FNT} for the equivalence of the first and third items when $K$ is a polyhedron, phrased in the language of operator systems. 

If $T \in \cB(H)^d_{sa}$ and $A$ is an \textit{invertible} affine transformation on $\bR^d$, then
\be\label{eq:affine}
\mathcal{W}(A(T)) = A(\mathcal{W}(T)),
\ee
where $A$ is applied to operator tuples in the natural way, as in \cite[\S 3]{PSS}. This also implies that for any compact convex set $K \subseteq \bR^d$,
\be\label{eq:affineminmax}
\Wmin{}(A(K)) = A(\Wmin{}(K)) \hspace{.4 in} \text{and} \hspace{.4 in} \Wmax{}(A(K)) = A(\Wmax{}(K)).
\ee
 Similarly, if $H \subseteq \bR^d$ is an affine subspace with orthogonal projection $P_H: \bR^d \to H$, then 
\be\label{eq:proj}
\Wmax{}(K) \subseteq \Wmin{}(L) \implies \Wmax{}(K \cap H) \subseteq \Wmin{}(P_H(L))
\ee
by \cite[Lemma 3.2]{PSS}. When our computations take place entirely in a proper affine subspace of $\bR^d$, we may then use (\ref{eq:affine}), (\ref{eq:affineminmax}), and (\ref{eq:proj}) to reduce the ambient space to $\bR^{n}$ for $n < d$. This allows us to prove results for all compact convex sets by focusing only on convex bodies (compact convex sets with nonempty interior). Roughly speaking, this corresponds to throwing out useless $0$ operators in a tuple $(T_1, \ldots, T_d, 0, \ldots, 0)$ to focus on $(T_1, \ldots, T_d)$ alone.

If one considers the graded product
\bes
\prod_{j=1}^N \mathcal{S}^j := \bigcup_{n=1}^\infty \prod_{j=1}^N \mathcal{S}^j_n
\ees
of matrix convex sets, then it is evident from (\ref{eq:wmaxdef}) that $\Wmax{}\left( \prod\limits_{j=1}^N K_i \right) = \prod\limits_{j=1}^N \Wmax{}(K_i)$. However, such a factorization generally does not exist for $\Wmin{}$. The root of the problem is that if $(A_1, \ldots, A_d)$ and $(B_1, \ldots, B_d)$ are normal tuples, then $(A_1, \ldots, A_d, B_1, \ldots, B_d)$ might fail to be normal, as the various $A_i$ and $B_j$ might not commute. In section \ref{sec:prodmin} we consider containment problems of the form
\bes
\prod_{j=1}^n \Wmin{}(K_i) \subseteq  \Wmin{}\left( \prod_{j=1}^N c_i \cdot K_i \right).
\ees
For symmetric $K_i$, it is possible to derive such containments from estimates concerning products of simplices and products of diamonds, as in Theorem \ref{thm:symmetricSDconsequences}. Since the matrix convex set $\mathcal{S}$ consisting of all $d$-tuples of matrix contractions may be written as $\prod\limits_{j=1}^d \Wmin{}(\overline{\bD})$, we can then obtain a dilation scale result in Corollary \ref{cor:contractionSDdilation} for tuples of contractions (see also Theorem \ref{thm:contraction_exp_dilation}). In contrast, for $K_i$ which are not necessarily symmetric, we show in Corollary \ref{cor:prodpossimplex} that estimates derived from a very slight modification of dilations in \cite{DDSS} cannot be improved.

In section \ref{sec:compmin}, we first consider the interaction between dilation theorems and compactness of operators. While we cannot guarantee that compactness is preserved in a dilation that comes from Proposition \ref{prop:introcara}, it does hold from Proposition \ref{prop:approxcompactdilation} that if $T$ is compact and $\mathcal{W}(T) \subseteq \Wmin{}(K)$, then there is a compact normal dilation $N$ of $T$ with spectrum in a neighborhood of $K$. Moreover, if a compact tuple $T$ has $\mathcal{W}(T) = \Wmin{}(K)$, then $K$ has the shape one would expect if it were assumed that $T$ is also normal, as in Theorem \ref{thm:compact_eigenvector_hunt}.  However, unlike in the finite-dimensional setting, the assumption that $T$ is minimal for its matrix range does not characterize $T$ up to unitary equivalence, and a minimal compact $T$ with $\mathcal{W}(T) = \Wmin{}(K)$ need not be normal. In particular, for many $K$ there exist uncountably many inequivalent, compact, minimal tuples with matrix range equal to $\Wmin{}(K)$, as in Corollary \ref{cor:uniqueness_issues}. Similarly, if $T$ is compact, a minimal summand $S$ of $T$ with the same matrix range might not exist, from Example \ref{ex:not_a_surprise}. These results indicated the need for additional assumptions in the theorems of \cite[\S 6]{DDSS}, which have now been corrected in response (see \cite{DDSScorrected}). In the restricted setting of $\Wmin{}$ sets and compact operators, we consider some alternative relaxations of the problem in Propositions \ref{prop:compression_min_Ik} and  \ref{prop:approxunitaryequivalencefix}, which we believe could be useful starting points for further study.

The remainder of section \ref{sec:compmin} concerns operator tuples which are not necessarily compact. A simple spectral theorem argument in Theorem \ref{thm:easy_spec_thm} shows that there is a minimal \textit{normal} $T$ for matrix range $\mathcal{W}(T) = \Wmin{}(K)$ if and only if $K$ satisfies a simple geometric condition: the isolated extreme points of $K$ are dense in the set of all extreme points of $K$. In this case, $T$ must be diagonal with eigenvalues at the isolated extreme points of $K$. However, if $K$ has at least three extreme points, then Corollary \ref{cor:non-compact_non-unique} shows that the same condition on isolated extreme points implies the existence of uncountably many \textit{non-normal} minimal tuples for matrix range $\Wmin{}(K)$. Finally, for any compact convex set $K$ with at least three extreme points, there is a tuple $T$ with $\mathcal{W}(T) = \Wmin{}(K)$ such that $T$ has no summand which is minimal for the same matrix range, and such that $T$ has no normal summands at all, as in Theorem \ref{thm:body_ball_covering}.

Finally, in section \ref{sec:scaled_containments}, we consider two matrix convex set containments that are demonstrated by explicit dilation procedures. First, we dilate tuples of contractions to normal tuples in Theorem \ref{thm:contraction_exp_dilation}, with a new norm bound (see also Corollary \ref{cor:contractionSDdilation}). Second, we give a lower bound for the matrix range of a universal tuple of anticommuting self-adjoint unitaries, using an explicit dilation procedure developed in Theorem \ref{thm:cont_to_ACaj} and Corollary \ref{cor:cubeball}.


\section{Products of Minimal Sets}\label{sec:prodmin}

As in \cite{PSS}, for nonempty compact convex sets $K$ and $L$ in Euclidean space, define
\bes
\theta(K) := \inf \{ C > 0: \Wmax{}(K) \subseteq C \cdot \Wmin{}(K) \}
\ees
and
\bes
\theta(K, L) := \inf\{ C > 0: \Wmax{}(K) \subseteq C \cdot \Wmin{}(L)\}.
\ees
Two disparate estimates
\be\label{eq:constsumm}
\theta([-1, 1]^d) = \sqrt{d} < d = \theta([0,1]^d)
\ee
were computed in \cite[Theorems 6.4 and 6.7]{PSS}, along with a non-uniform version of the first equality,
\be\label{eq:nonuniformconstsumm}
\Wmax{}([-1,1]^d) \subseteq \Wmin{}\left(\prod_{j=1}^d [-a_j, a_j] \right) \iff \sum_{j=1}^d a_j^{-2} \leq 1.
\ee 
Equation (\ref{eq:constsumm}) is equivalent to the following dilation results.
\begin{enumerate}
\item\label{item:casesa} If $(X_1, \ldots, X_d) \in \cB(H)^d_{sa}$ is a tuple of \textit{self-adjoint} contractions, then there exists a dilation tuple $(M_1, \ldots, M_d)$ of commuting self-adjoint operators with norm $||M_i|| \leq \sqrt{d}$, and $\sqrt{d}$ is the optimal constant.
\item\label{item:casepos} If $(P_1, \ldots, P_d) \in \cB(H)^d_{sa}$ is a tuple of \textit{positive} contractions, then there exists a dilation tuple $(N_1, \ldots, N_d)$ of commuting positive operators with norm $||N_i|| \leq d$, and $d$ is the optimal constant.
\end{enumerate}

Both computations are paired with explicit dilation procedures. The disparity between the two constants emphasizes the fact that, when dilating self-adjoint contractions to self-adjoint operators which commute, the preservation of another relation among the contractions (namely, positivity) significantly alters the norm bound one can achieve. Consistent with this idea, the dilation procedure in \cite[Theorem 6.7]{PSS} which demonstrates (\ref{item:casesa}) begins by replacing each $X_i$ with a self-adjoint unitary. Therefore, the dilation procedure is generally not able to preserve additional properties of the tuple. Namely, it cannot
\begin{itemize}
\item preserve compactness of each $X_i$, or
\item preserve the satisfaction of other linear or non-linear inequalities by the $X_i$ (which are independent from $-1 \leq X_i \leq 1$), or
\item successfully use the fact that \textit{some} of the $X_i$ might already commute to lower the norm of the dilation in this special case
\end{itemize}
without modification. In contrast, an earlier explicit dilation procedure that demonstrates $\theta([-1,1]^d) \leq d$ in \cite[\S 7]{DDSS}, while far from achieving the optimal constant, \textit{does} preserve compactness information, and it simultaneously demonstrates multiple containments of the form $\Wmax{}(K) \subseteq \Wmin{}(L)$. In this section, we expand upon the third bullet point, in pursuit of the following problem.

\begin{problem}
Compute when $\prod\limits_{j=1}^d \Wmin{}(K_i) \subseteq \Wmin{}\left( L \right)$. Namely, if a large tuple $M$ consists of smaller subtuples $M^{[i]} = (M^{[i]}_1, \ldots, M^{[i]}_{n_i})$, where each $M^{[i]}$ is normal with joint spectrum $\sigma(M^{[i]}) \subseteq K_i$, does $M$ admit a normal dilation $N$ with joint spectrum in $L$?
\end{problem}

In particular, we show that when the $K_i$ are symmetric and $L$ is the product of (perhaps distinct) multiples of $K_i$, containment theorems follow from seemingly unrelated dilation constants. These constants are defined in reference to products of simplices and products of diamonds.

\begin{definition}
The \textit{standard simplex} $\Delta_n$ in $\bR^n$ refers to the convex hull of $0$ and the standard basis vectors $e_1, \ldots, e_n$. The corresponding \textit{standard diamond} $\Diam_n$ is the $\ell^1$ unit ball in $\bR^n$, i.e., the smallest symmetric convex set containing $\Delta_n$.
\end{definition}

\begin{definition}
Call a tuple $(a_1, \ldots, a_d)$ of positive numbers an \textit{SD-tuple} if it holds that for any $n \in \bZ^+$,
\be\label{eq:SDdefcontainment}
\Wmax{}\left( \Delta_n^d \right) \subseteq \Wmin{}\left(\prod_{j=1}^d a_j \cdot \Diam_n \right).
\ee
Similarly, let 
\bes \begin{aligned}
\USD{d} 	&:= \sup_{n \in \bZ^+} \theta( \Delta_n^d, \Diam_n^d) \\
			&= \inf \{C > 0: (C, \ldots, C) \text{ is an SD-tuple of length } d\}
\end{aligned} \ees 
be called the \textit{uniform SD-constant for }$d$-\textit{tuples}.
\end{definition}

Dilation techniques from \cite{DDSS} imply that SD-tuples exist and that $\USD{d}$ is finite, and such results will be clarified later in this section. The most common use of SD-tuples will come in the following form.

\begin{proposition}\label{prop:SDdilation}
Fix $n_1, \ldots, n_d \in \bZ^+$ and let $P$ be a tuple of self-adjoint operators $P^{[i]}_j$, $1 \leq i \leq d$, $1 \leq j \leq n_i$, on $\cB(H)$  such that $W_1(P) \subseteq \left( \prod\limits_{i = 1}^d \Delta_{n_i} \right)$. (That is, $P^{[i]}_j \geq 0$, and for each $i$, $\sum\limits_{j=1}^{n_i} P^{[i]}_j \leq I$.) If $(a_1, \ldots, a_d)$ is an SD-tuple, then there exists a dilation $Q \succ P$ consisting of self-adjoint operators with the following properties.
\begin{enumerate}
\item  Each $Q^{[i]}_j$ has $\sigma(Q^{[i]}_j) \subseteq \{ -a_i, 0, a_i \}$.
\item The operators $Q^{[i]}_j$ all commute with each other.
\item If $i$ is fixed, then distinct members of the tuple $Q^{[i]} = (Q^{[i]}_1, \ldots, Q^{[i]}_{n_i})$ are orthogonal.
\item If $\cH$ is finite-dimensional, then $Q$ acts on a finite-dimensional space as well.
\end{enumerate}
\end{proposition}
\begin{proof}
By introducing the $0$ operator into tuples if necessary, we may assume that $n_1, \ldots, n_d$ are equal to a fixed $n$. Since $\Wmax{}\left( \Delta_n^d \right) \subseteq \Wmin{}\left(\prod\limits_{j=1}^d a_j \cdot \Diam_n \right)$ holds, Proposition \ref{prop:introcara} shows there is a normal dilation $Q$ of $P$ whose joint spectrum lies in the extreme points of $\prod\limits_{j=1}^d a_j \cdot \Diam_n$, where $Q$ acts on a finite-dimensional space if $\cH$ is finite-dimensional. The extreme points of $\prod\limits_{j=1}^d a_j \cdot \Diam_n$ are positioned exactly so that the remaining properties (1)-(3) also hold.
\end{proof}
\begin{remark}
It is important to note that the properties listed do \textit{not} imply that the operators $Q^{[i]}_j$ are positive, or that all products of distinct $Q^{[i]}_j$ are zero. We abuse notation somewhat by letting $Q = (Q^{[1]}, \ldots, Q^{[d]})$ denote the \lq\lq conjoined tuple\rq\rq\hspace{0pt} consisting of all the matrices $Q^{[i]}_j$, ordered by $i$ first and $j$ second.
\end{remark}

\begin{theorem}\label{thm:symmetricSDconsequences}
Let $(a_1, \ldots, a_d)$ be an SD-tuple, and for each $1 \leq i \leq d$, let $K_i \subseteq \bR^{n_i}$ be a compact convex set with $K_i = -K_i$. Then
\be\label{eq:symmSDdilationresult}
\prod_{i=1}^d \Wmin{}(K_i) \subseteq \Wmin{}\left( \prod_{i=1}^d a_i K_i \right)
\ee
holds. Consequently, it also holds that
\be
\Wmax{}\left( \prod_{i=1}^d K_i \right) \subseteq \Wmin{}\left( \prod_{i=1}^d \theta(K_i) a_i K_i \right)
\ee
and
\be
\theta\left( \prod_{i=1}^d K_i \right) \leq \inf _{a \text{ is SD}} \left[ \max_{1 \leq i \leq d} \hspace{4 pt} a_i \theta(K_i) \right] \leq  \USD{d} \cdot \max_{1 \leq i \leq d} \theta(K_i).
\ee
\end{theorem}
\begin{proof}
Fix $m \in \bZ^+$. For each $i \in \{1, \ldots, d\}$, let $N^{[i]} = (N^{[i]}_1, \ldots, N^{[i]}_{n_i})$ be a normal tuple of self-adjoint $m \times m$ matrices with joint spectrum satisfying $\sigma(N^{[i]}) \subseteq K_i$. The normal tuple $N^{[i]}$ admits a joint diagonalization, so we may specify a tuple $P^{[i]} =  (P^{[i]}_1, \ldots, P^{[i]}_m)$ of mutually orthogonal projections with $\sum\limits_{k=1}^m P^{[i]}_k = I_m$ and a collection of eigenvalues $\lambda^{[i]}_{j,k}$ such that
\bes
N^{[i]}_j = \sum_{k=1}^m \lambda^{[i]}_{j,k} P^{[i]}_k
\ees
holds. By assumption, for each $i$ and $k$, the tuple $(\lambda^{[i]}_{1,k}, \ldots, \lambda^{[i]}_{n_i,k})$ belongs to $K_i$. 

Consider the conjoined tuple $P = (P^{[1]}, \ldots, P^{[d]})$, which has $\mathcal{W}_1(P) \subseteq \Delta_m^d$, so that we may form a dilation $Q = (Q^{[1]}, \ldots, Q^{[d]})$ guaranteed by Proposition \ref{prop:SDdilation}. All of the matrices $Q^{[i]}_k$ commute with each other, and if $i$ is fixed but $k \not= l$, it follows that $Q^{[i]}_k Q^{[i]}_l = 0$. Finally, each matrix $Q^{[i]}_k$ has $\sigma(Q^{[i]}_k) \subseteq \{-a_i, 0, a_i\}$. Therefore, for each fixed $i$, the matrices
\be\label{eq:sortofdiag}
M^{[i]}_j := \sum_{k=1}^m \lambda^{[i]}_{j,k} Q^{[i]}_k
\ee
are essentially given in jointly diagonalized form, and the tuple $M^{[i]} = (M^{[i]}_1, \ldots, M^{[i]}_{n_i})$ is normal with joint spectrum contained in $-a_i K_i \cup \{0\} \cup a_i K_i = a_i K_i$. Further, for any choices of $i$ and $j$, the self-adjoint matrices $M^{[i]}_j$ commute with each other, so the conjoined tuple $M = (M^{[1]}, \ldots, M^{[d]})$ is also normal. Finally, the joint spectrum of $M$ is contained in $\prod\limits_{i=1}^d \sigma(M^{[i]})$, which is contained in $\prod\limits_{i=1}^d a_i K_i$, so (\ref{eq:symmSDdilationresult}) holds. The remaining identities then follow from the equality $\Wmax{}\left( \prod\limits_{i=1}^d K_i\right) = \prod\limits_{i=1}^d \Wmax{}(K_i)$ and the definitions of $\theta(K_i)$ and $\USD{d}$.
\end{proof}

A consequence of this result is that SD-tuples may be defined with reference only to diamonds, as opposed to both simplices and diamonds.

\begin{corollary}
A tuple $(a_1, \ldots, a_d)$ of positive numbers is an SD-tuple if and only if for each $n \in \bZ^+$, 
\be\label{eq:diamondonlySD}
\prod_{j=1}^d \Wmin{}\left(\Diam_n \right) \subseteq \Wmin{}\left(\prod_{j=1}^d a_j \cdot \Diam_n \right).
\ee
\end{corollary}
\begin{proof}
The forward direction is given by Theorem \ref{thm:symmetricSDconsequences} for the choice of symmetric set $K_i = \Diam_n$. For the converse, note that (\ref{eq:diamondonlySD}) directly implies (\ref{eq:SDdefcontainment}), as
\bes
\Wmax{}\left( \Delta_n^d \right) = \prod_{j=1}^d \Wmax{}(\Delta_n) = \prod_{j=1}^d \Wmin{}(\Delta_n) \subseteq \prod_{j=1}^d \Wmin{}(\Diam_n) \subseteq \Wmin{}\left(\prod_{j=1}^d a_j \cdot \Diam_n \right).
\ees
\end{proof}

We may also dilate tuples of contractions using the previous results. Let $\overline{\bD}$ be the closed unit disk and fix any $Y \in \Wmin{}(\overline{\bD})$, where we may choose to view $Y$ as a single matrix (instead of two self-adjoint matrices). Since $Y$ has a normal dilation $N$ with \lq\lq joint\rq\rq\hspace{0pt} spectrum in $\bS^1 \subseteq \overline{\bD}$, it follows that $||N|| \leq 1$ and consequently $||Y|| \leq 1$. On the other hand, if $X$ is any matrix with $||X|| \leq 1$, then the Halmos dilation procedure
\be\label{eq:SDhalmosdilationdefinition}
X \textrm{ is a contraction } \implies U := \begin{pmatrix} X & \sqrt{I - XX^*} \\ \sqrt{I - X^*X} & -X^* \end{pmatrix} \text{ is a unitary dilation}
\ee
of \cite{halmos} shows that $X \in \Wmin{}(\overline{\bD})$. It therefore follows that $\Wmin{}(\overline{\bD})$ is precisely the collection of (not necessarily self-adjoint) matrix contractions. Applying the graded product, we obtain the following simple fact:
\be\label{eq:IcantbelieveImissedthisonthefirsttimearound}
\prod_{j=1}^d \Wmin{}(\overline{\bD}) = \{T \in \Md: \text{for each } i \in \{1, \ldots, d\}, ||T_i|| \leq 1\}.
\ee

\begin{corollary}\label{cor:contractionSDdilation}
Let $H$ be a Hilbert space of any dimension, and let $T \in \cB(H)^d$ be a tuple of (not necessarily self-adjoint) contractions. Then for any SD-tuple $(a_1, \ldots, a_d)$, it holds that
\begin{itemize}
\item $\mathcal{W}(T) \subseteq \Wmin{}\left(\prod\limits_{i=1}^d a_i \overline{\mathbb{D}}\right)$, and
\item there exists a normal tuple $N$ which dilates $T$ and has $||N_i|| \leq a_i$ for each $i$. 
\end{itemize}
In particular, we may choose $a_i = \USD{d}$, and if $\cH$ is finite-dimensional, we may choose $N$ which acts on a finite-dimensional space.
\end{corollary}
\begin{proof}
Since $T$ is a tuple of contractions, any tuple $A \in \mathcal{W}(T)$ consists of matrix contractions and is therefore contained in $\prod\limits_{j=1}^d \Wmin{}(\overline{\bD})$ by (\ref{eq:IcantbelieveImissedthisonthefirsttimearound}). By Theorem \ref{thm:symmetricSDconsequences}, it follows that $A \in \Wmin{}\left(\prod\limits_{i=1}^d a_i \overline{\mathbb{D}}\right)$. Since we then have $\mathcal{W}(T) \subseteq \Wmin{}\left(\prod\limits_{i=1}^d a_i \overline{\mathbb{D}}\right)$, applying Proposition \ref{prop:introcara} finishes the proof.
\end{proof}

See Theorem \ref{thm:contraction_exp_dilation} for an explicit dilation procedure that begins with a tuple $T \in \cB(H)^d$ of contractions and ends with a normal tuple $N$ satisfying $||N_i|| \leq \sqrt{2d}$ for each $i$, where the constant $\sqrt{2d}$ is not necessarily optimal. The optimal dilation scale of \textit{self-adjoint} contractions is known from (\ref{eq:nonuniformconstsumm}), allowing us to place bounds on the collection of SD-tuples.

\begin{corollary}
If $(a_1, \ldots, a_d)$ is an SD-tuple, then $\sum\limits_{j=1}^d a_j^{-2} \leq 1$. Consequently, $\USD{d} \geq \sqrt{d}$.
\end{corollary}
\begin{proof}
If $(a_1, \ldots, a_d)$ is an SD-tuple, then by Theorem \ref{thm:symmetricSDconsequences}, it holds that
\bes
\prod_{j=1}^d \Wmin{}([-1,1]) \subseteq \Wmin{}\left( \prod_{j=1}^d [-a_j, a_j] \right).
\ees
However, the minimal and maximal matrix convex set over an interval are identical, so we have
\bes
\Wmax{}([-1,1]^d) = \prod_{j=1}^d \Wmax{}([-1,1]) = \prod_{j=1}^d \Wmin{}([-1,1]) \subseteq  \Wmin{}\left( \prod_{j=1}^d [-a_j, a_j] \right).
\ees
Finally, from (\ref{eq:nonuniformconstsumm}) we conclude that $\sum\limits_{j=1}^d a_j^{-2} \leq 1$.
\end{proof}

Following the orthogonal case of \cite[Theorem 7.7]{DDSS}, given $a_1, \ldots, a_d > 0$ with $\sum\limits_{i=1}^d a_i^{-1} = 1$, we have that for any orthonormal system $P_i = v_i v_i^*$ of rank one projections such that $\sum\limits_{i=1}^d P_i = I_d$ (that is, $v_1, \ldots, v_d$ form a basis of $\bC^d$), the unit vector $w := a_1^{-1/2}v_1 + \ldots + a_d^{-1/2}v_d$ has $\langle a_i P_i w, w \rangle = 1$. Applying a change of basis, we find that for any tuple $(a_1, \ldots, a_d)$ of positive numbers with $\sum\limits_{i=1}^d a_i^{-1} = 1$, there exist $d \times d$ matrices $Q_1, \ldots, Q_d$ with
\be\label{eq:thetensors}
Q_i = Q_i^*, \hspace{.15 in} \text{Rank}(Q_i) = 1, \hspace{.15 in} \sigma(Q_i) = \{0, a_i\}, \hspace{.15 in} Q_i Q_j = 0 \text{ for } i \not= j, \hspace{.15 in} Q_i = \begin{pmatrix} 1 & * \\ * & * \end{pmatrix}.
\ee
It follows that in this circumstance, the dilation technique of \cite[Theorem 7.7]{DDSS} replaces a self-adjoint operator $T_i$ with $T_i \otimes Q_i$, so that $(T_1 \otimes Q_1, \ldots, T_d \otimes Q_d)$ is a normal tuple. We will use such a dilation, where we replace each $T_i$ with a tuple containing multiple operators.

\begin{theorem}\label{thm:positivescaling}
 For each $1 \leq i \leq d$, let $K_i \subseteq \bR^{n_i}$ be a compact convex set with $0 \in K_i$ for each $i$. For positive numbers $t_1, \ldots, t_d$, let $L[t_1, \ldots, t_d] \subseteq \prod\limits_{t=1}^d t_i K_i$ be the convex hull of all the sets $\{0\} \times t_{i} K_{i} \times \{0\}$, $1 \leq i \leq d$, where $0$ denotes a tuple of zeroes of the appropriate size. Then for any positive numbers $a_1, \ldots, a_d$ with $\sum\limits_{i=1}^d a_i^{-1} \leq 1$, it follows that
\be\label{eq:prodminK_i}
\prod_{i=1}^n \Wmin{}\left( K_i\right) \subseteq \Wmin{}(L[a_1, \ldots, a_d]) \subseteq \Wmin{}\left( \prod_{i=1}^n a_i K_i \right)
\ee
holds. Consequently, it also holds that
\be\label{eq:prodmaxK_i}
\Wmax{}\left(\prod_{i=1}^n K_i \right) \subseteq \Wmin{}(L[\theta(K_1) a_1, \ldots, \theta(K_d) a_d]) \subseteq \Wmin{}\left( \prod_{i=1}^n \theta(K_i) a_i K_i \right)
\ee
and
\be\label{eq:posbadtheta}
\theta\left( \prod_{i=1}^n K_i  \right) \leq \inf_{a_i > 0, \hspace{2 pt} \sum a_i^{-1} \leq 1} \left[ \max_{1 \leq i \leq d} \hspace{4 pt} \theta(K_i) a_i  \right] \leq d \cdot \max_{1 \leq i \leq d} \theta(K_i).
\ee
\end{theorem}
\begin{proof}
Since $0 \in K_i$, we may shrink the $a_i$ so that $\sum\limits_{i=1}^d a_i^{-1} = 1$. From (\ref{eq:thetensors}), there are mutually orthogonal projections $P_1, \ldots, P_d \in M_d$ such that $Q_i := a_i P_i$ has entry $1$ in the top-left corner. For each $i$, let $M^{[i]}\in \cB(H)^{n_i}_{sa}$ be a normal tuple with $\sigma(M^{[i]}) \subseteq K_i$. Dilate each operator $M^{[i]}_j$ by choosing 
\bes
N^{[i]}_j := M^{[i]}_j \otimes Q_i.
\ees
Then the conjoined tuple $N = (N^{[1]}, \ldots, N^{[d]})$ is a normal dilation of $M = (M^{[1]}, \ldots, M^{[d]})$. The spectrum of each $N^{[i]}$ is in $a_i K_i$, and moreover $N^{[i]}_j N^{[k]}_l = 0$ if $i \not = k$. It follows that the joint spectrum of $N$ is contained in $L([a_1, \ldots, a_d])$, and we may conclude that (\ref{eq:prodminK_i}) holds. Next, $(\ref{eq:prodmaxK_i})$ follows from $(\ref{eq:prodminK_i})$, as any tuple of operators with numerical range in $K_i$ admits a normal dilation with joint spectrum in $\Wmin{}(\theta(K_i) K_i)$, by definition. Finally, (\ref{eq:posbadtheta}) follows immediately from $(\ref{eq:prodmaxK_i})$.
\end{proof}

\begin{corollary}
If $(a_1, \ldots, a_d)$ is a tuple of positive numbers with $\sum\limits_{i=1}^d a_i^{-1} \leq 1$, then $(a_1, \ldots, a_d)$ is an SD-tuple. Consequently, $\USD{d} \leq d$.
\end{corollary}
\begin{proof}
Theorem \ref{thm:positivescaling} shows that for any $n \in \bZ^+$, 
\bes
\Wmax{}(\Delta_n^d) \subseteq \Wmin{}\left( \prod_{j=1}^d \theta(\Delta_n) a_j \Delta_n \right) = \Wmin{}\left(\prod_{j=1}^d a_j \Delta_n \right) \subseteq \Wmin{}\left(\prod_{j=1}^d a_j \Diam_n \right),
\ees
so by definition $(a_1, \ldots, a_d)$ is an SD-tuple.
\end{proof}

All together, we have that for positive tuples $(a_1, \ldots, a_d)$, 
\be
\frac{1}{a_1} + \ldots + \frac{1}{a_d} \leq 1 \implies (a_1, \ldots, a_d) \text{ is an SD-tuple } \implies \frac{1}{a_1^2} + \ldots + \frac{1}{a_d^2} \leq 1,
\ee
and in particular
\be\label{eq:USD_bounds_sqrtandd}
\sqrt{d} \leq \mathcal{U}(d) \leq d.
\ee
Analogous to the computation (\ref{eq:constsumm}), we suspect that if one of the implications is an equivalence, then SD-tuples may be characterized by the identity $\sum\limits_{j=1}^d a_j^{-2} \leq 1$. However, the estimates of Theorem \ref{thm:positivescaling} are optimal for certain positive sets, such as $[0,1]^d$, once again emphasizing that dilation problems concerning symmetric sets are (or have the potential to be) considerably more flexible than those concerning positive sets. Recall that the difficulties of positive-to-positive dilation were abstracted in \cite{PSS} to \textit{simplex-pointed} sets.

\begin{definition}\label{def:simplex-pointed}
A convex body $K \subseteq \bR^d$ is \textit{simplex-pointed} at $x \in K$ if there exists a basis $v_1, \ldots, v_d$ of $\bR^d$ such that the convex hull of $x, x+v_1, \ldots, x+v_d$ is contained in $K$ and is a neighborhood of $x$ in $K$, where $K$ is equipped with the relative topology from $\bR^d$. Equivalently, there is an invertible affine transformation $A$ on $\bR^d$ such that $A(K) \subseteq [0, \infty)^d$, $A(x) = 0$, and $A(K)$ includes a neighborhood of $0$ in $[0, \infty)^d$.
\end{definition}

The following theorem expands upon the techniques in \cite{PSS} and gives a partial answer to \cite[Problem 8.6]{PSS}: if $\Wmax{}(K) \subseteq \Wmin{}(L)$, under what circumstances can we conclude that there is a simplex $\Pi$ with $K \subseteq \Pi \subseteq L$?

\begin{theorem}\label{thm:simplexcontainment}
Let $K \subseteq \bR^d$ be a convex body which is simplex-pointed at $x \in K$ with $v_1, \ldots, v_d$ as in Definition \ref{def:simplex-pointed}. Let $L$ be another convex body which is simplex-pointed at the same point $x$ with the same vector data $v_1, \ldots, v_d$, and suppose that $\Wmax{2}(K) \subseteq \Wmin{2}(L)$. If
\be\label{eq:simplexdirections}
t_i := \max \{t \geq 1: x + t v_i \in L \},
\ee
then the convex hull of $x, x+t_1 v_1, \ldots, x+t_d v_d$ is a simplex $\Pi$ with $K \subseteq \Pi \subseteq L$.

Consequently, if $K$ is simplex-pointed at $x = 0$ with vector data $v_1, \ldots, v_d$, and 
\bes
\mathcal{S} := \{\text{conv}(0, s_1v_1, \ldots, s_d v_d): s_1, \ldots, s_d \geq 1 \},
\ees
then
\bes
\theta(K) = \min \{C \geq 1: \text{there exists } \Pi \in \mathcal{S} \text{ with }  K \subseteq \Pi \subseteq C \cdot K \}.
\ees
\end{theorem}
\begin{proof}
After an invertible affine transformation (linear if $x = 0$), we may suppose that $K$ and $L$ are contained in $[0, \infty)^d$, $x = 0$, and the standard simplex $\Delta_d$ is contained in $K$ and $L$. Letting $e_1, \ldots, e_d$ denote the standard basis vectors in $\bR^d$, we wish to prove that if
\be\label{eq:sumteesplease}
t_i := \max \{t \geq 1: t e_i \in L\},
\ee
then the convex hull $\Pi$ of $0, t_1 e_1, \ldots, t_d e_d$ has $K \subseteq \Pi$ (as the containment $\Pi \subseteq L$ is trivial). 

Fix any interior point $c$ of $K$, and let $q = (d-1) \cdot \max\limits_{1 \leq i \leq d }\{c_i\}$. Let $\varepsilon > 0$ be small enough that the $q\varepsilon$-neighborhood $N_{q \varepsilon}$ of the line segment $\textrm{conv}(0, c)$ has $N_{q\varepsilon} \cap [0,\infty)^d \subseteq K$. Next, let $P = (P_1, \ldots, P_d)$ be a tuple of $2 \times 2$ rank 1 projections such that $\textrm{Ran}(P_i) \cap \textrm{Ran}(P_j) = \{0\}$ if $i \not= j$, but $||P_i - P_j|| < \varepsilon$. Setting $X = (c_1 P_1, \ldots, c_d P_d)$ and $Q = (c_1 P_1, c_2 P_1, \ldots, c_d P_1)$, we have that $\mathcal{W}_1(X) \subseteq [0, \infty)^d$ is within the $q\varepsilon$-neighborhood of $\mathcal{W}_1(Q) = \text{conv}(0, c)$. This implies two key facts. First, we have that $\mathcal{W}_1(X) \subseteq K$. Second, applying a vector state corresponding to a unit vector in $\textrm{Ran}(P_1)$ shows that
\be\label{eq:approximatingc}
\exists v \in \mathcal{W}_1(X) \text{ with } ||v - c|| < q\varepsilon.
\ee

Since $\mathcal{W}_1(X) \subseteq K$ and $X$ is a tuple of $2 \times 2$ matrices, it follows that $X \in \Wmax{2}(K) \subseteq \Wmin{2}(L)$, and $X$ admits a normal dilation with joint spectrum in $L \subseteq [0,\infty)^d$. By design, the tuple $X$ has the property that if $i \not= j$ and $T$ is a matrix with $0 \leq T \leq X_i$ and $0 \leq T \leq X_j$, then $T = 0$. Therefore, by \cite[Lemma 6.3]{PSS}, there is a possibly distinct normal dilation $Z$ of $X$ such that $\sigma(Z) \subseteq L \cup \{0\} =  L$ and $Z_i Z_j = 0$ for $i \not= j$. That is, $\sigma(Z) \subseteq L$ consists of points which have at most one nonzero coordinate. From (\ref{eq:sumteesplease}) and the definition of $\Pi$ immediately thereafter, we have that $\sigma(Z) \subseteq \Pi$ and $X \in \Wmin{}(\Pi)$. We then conclude from (\ref{eq:approximatingc}) that there is a point $v \in \mathcal{W}_1(X) \subseteq \Wmin{1}(\Pi) = \Pi$ such that $||v - c|| < q \varepsilon$. Since we may repeat the procedure for arbitrarily small $\varepsilon$, and $\Pi$ (which does not depend on $\varepsilon$) is closed, we have that $c \in \Pi$. Next, $c$ was an arbitrary interior point of $K$, so the interior of $K$ is contained in $\Pi$. Finally, $K$ is a convex body, so $K$ is the closure of its interior, and it follows that $K \subseteq \Pi$.
\end{proof}

Theorem \ref{thm:simplexcontainment} directly generalizes the computations of dilation scale in \cite[Theorem 6.4]{PSS}. Moreover, it also implies \cite[Theorem 8.8]{PSS}, as the technical condition given therein shows that $K$ is a simplex-pointed set, $\Delta$ is a simplex containing $K$ that emanates from $x$ in the same direction as the given vector data $v_1, \ldots, v_d$, and $\Delta$ is minimal among simplices which contain $K$. Thus, we find that the particularly precise perturbation method used in the proof of \cite[Theorem 8.8]{PSS} is ultimately not necessary. Applied to products of simplices, Theorem \ref{thm:simplexcontainment} implies that the estimates of Theorem \ref{thm:positivescaling} are optimal when each $K_i$ is a simplex.

\begin{corollary}\label{cor:prodpossimplex}
For each $1 \leq i \leq d$, let $n_i \geq 1$ and let $K_i$ be a simplex of dimension $n_i$ with $0 \in K_i$. Then for positive scalars $a_1, \ldots, a_d$, $\Wmax{}\left(\prod\limits_{i=1}^d K_i \right) = \prod\limits_{i=1}^d \Wmin{}\left(K_i\right)$ is contained in $\Wmin{}\left(\prod\limits_{i=1}^d a_i K_i \right)$ if and only if $\sum\limits_{i=1}^d a_i^{-1} \leq 1$. In particular, $\theta\left(\prod\limits_{i=1}^d K_i\right) = d$.
\end{corollary}
\begin{proof}
The forward direction is proved in Theorem \ref{thm:positivescaling}. For the converse, we may assume $K_i$ is the standard simplex by applying a linear transformation and restriction to a proper subspace, if necessary. Since $\prod\limits_{i=1}^d K_i$ is then simplex-pointed at $x = 0$ with vector data given by the standard basis $e_1, \ldots, e_{n_1+\ldots+n_d}$, Theorem \ref{thm:simplexcontainment} shows that if $\Wmax{}\left( \prod\limits_{i=1}^d K_i \right) \subseteq \Wmin{}\left(\prod\limits_{i=1}^d a_i K_i \right)$, then the simplex $L$ spanned by $0, a_1 e_1, \ldots, a_1 e_{n_1}$, $a_2 e_{n_1+1}, \ldots, a_2 e_{n_1+n_2}$, $\ldots, a_d e_{n_1+\ldots+n_{d-1}+1}, \ldots, a_d e_{n_1+\ldots+n_d}$ must have $\prod\limits_{i=1}^d K_i \subseteq L$. This implies $\sum\limits_{i=1}^d a_i^{-1} \leq 1$ by consideration of the point $(1, 0_{n_1-1}, 1, 0_{n_2-1}, \ldots, 1, 0_{n_d-1})$.
\end{proof}


\section{Compactness and Minimality}\label{sec:compmin}

In this section, we primarily consider problems related to the matrix ranges of compact operator tuples. First, recall that Proposition \ref{prop:introcara} demonstrates that if $T$ is an operator tuple with $\mathcal{W}(T) \subseteq \Wmin{}(K)$, then $T$ has a normal dilation $N$ with joint spectrum in $K$. Moreover, if $T$ is a matrix tuple, then there exists a choice of $N$ which is also a matrix tuple. However, no claim is made about the preservation of compactness if $T$ acts on an infinite-dimensional space. Below we prove an approximate result in this direction.

\begin{proposition}\label{prop:approxcompactdilation}
Let $T \in \cK(H)^d$ be a tuple of compact operators acting on an infinite-dimensional space, with $\mathcal{W}(T) \subseteq \Wmin{}(K)$. Then for any $\varepsilon > 0$, there exists a compact normal dilation $N$ of $T$ such that $N$ decomposes as a direct sum $A \oplus B$ with $||A_i|| < \varepsilon$ and $\sigma(B) \subseteq (1 + \varepsilon) K$.
\end{proposition}

\begin{proof}
We may assume that $T$ consists of compact self-adjoint operators. Since $T$ acts on an infinite-dimensional space, $0$ belongs to $\mathcal{W}_1(T) \subseteq K$. Given $\delta > 0$, let $F_1, \ldots, F_d$ be self-adjoint finite rank operators with $||T_i - F_i|| < \delta/2$. Then if $P$ is the finite-dimensional projection onto the sum of the kernels and cokernels of the $F_i$, it follows from conjugation by $P$ that $||P F_i P - PT_iP|| = ||F_i - P T_i P || < \delta/2$. Therefore, $||T_i - PT_iP|| < \delta$. 

First, consider the tuple $X = (T_1 - PT_1P, \ldots, T_d - PT_dP)$. The operators in the tuple are compact, but since $P$ need not be a reducing subspace for $T$, $\mathcal{W}_1(X)$ might not be contained in $K$. However, since $||T_i - P T_i P|| < \delta$, the dilation technique of the second half of the proof of \cite[Theorem 7.4]{DDSS}, which preserves compactness of self-adjoint operators, produces a compact normal dilation $R$ of $X$ with $||R_i|| < d \delta$.

Next, we dilate $Y = (PT_1P, \ldots, PT_dP)$, which is unitarily equivalent to $(M_1 \oplus 0, \ldots, M_d \oplus 0)$ for some matrix tuple $M$ with $M \in \mathcal{W}(T) \subseteq \Wmin{}(K)$. Since $M$ is a matrix tuple, it admits a normal matrix dilation with spectrum in $K$ by Proposition \ref{prop:introcara}, so $Y$ admits a finite rank (hence compact) normal dilation $S$ with $\sigma(S) \subseteq K \cup \{0\} = K$.

Finally, we combine the two dilations. The tuples $R$ and $S$ might act on different spaces, but this is easily remedied with the addition of zero summands, if necessary. Now, $R +S$ is a dilation of $T$, but it might not be a normal tuple, as the $R_i$ and $S_j$ might not commute. Given $a, b > 0$ with $\frac{1}{a} + \frac{1}{b} = 1$, let $Q_1$ and $Q_2$ be positive $2 \times 2$ matrices with entry $1$ in the top left corner such that $\sigma(Q_1) = \{0, a\}$, $\sigma(Q_2) = \{0, b\}$, and $Q_1 Q_2 = Q_2 Q_1 = 0$, following (\ref{eq:thetensors}). Let $A_i = R_i \otimes Q_1$ and $B_j = S_j \otimes Q_2$, so that $A_i B_j = 0 = B_j A_i$, and the orthogonal sum $N = A + B$ is a normal dilation of $T$. We have that $||A_i|| < a d \delta$ for each $i$, and $\sigma(B) \subseteq bK$. Given $\varepsilon > 0$, we complete the proof by noting that we could have chosen $b = 1 + \varepsilon$, $a = \frac{1}{1 - 1/b}$, and $\delta = \frac{\varepsilon}{ad}$.
\end{proof}

Combining Proposition \ref{prop:approxcompactdilation} with (\ref{eq:nonuniformconstsumm}) (from \cite[Theorem 6.7]{PSS}), we find that the dilation of compact self-adjoint contractions may be bounded in the following sense.

\begin{corollary}\label{cor:approxcompcontractions}
Let $(a_1, \ldots, a_d)$ be a tuple of positive numbers such that $\sum\limits_{i=1}^d a_i^{-2} < 1$. Then given any tuple $T \in \cK(H)^d_{sa}$ of compact self-adjoint contractions, there exists a normal dilation $N$ of compact self-adjoint operators with $||N_i|| \leq a_i$.
\end{corollary}

Similarly, there are bounds on the dilation of compact contractions which are not necessarily self-adjoint.

\begin{corollary}\label{cor:approxSDcompnotsacontractions}
Let $(a_1, \ldots, a_d)$ be a tuple of positive numbers such that for some $\varepsilon > 0$, $(1+\varepsilon) \cdot (a_1, \ldots, a_d)$ is an SD-tuple. Then given any tuple $T \in \cK(H)^d$ of compact contractions, there exists a normal dilation $N$ of compact operators with $||N_i|| \leq a_i$.
\end{corollary}

It would be interesting to know if compact-to-compact dilation could be achieved without the approximation of spectrum used in Proposition \ref{prop:approxcompactdilation}. This would be useful even in the particular case of Corollary \ref{cor:approxcompcontractions}, as compact-to-compact dilation without perturbation of bounds would demonstrate that the use of Halmos dilation
\be\label{eq:halmosdilationdefinition} \ba
X \textrm{ is a (self-adjoint)} &\textrm{ contraction } \implies \\
				&U := \begin{pmatrix} X & \sqrt{I - XX^*} \\ \sqrt{I - X^*X} & -X^* \end{pmatrix} \text{ is a (self-adjoint) unitary dilation }
\ea \ee
in \cite[Theorem 6.7]{PSS} is not optimal, as it immediately removes compactness. On the other hand, if $\mathcal{W}(T) \subseteq \Wmin{}(K)$, then for the diagonal operator tuple $N$ with eigenvalues at all points of $K$, there is a UCP map sending $N_i$ to $T_i$. If a compact normal dilation of $T$ with joint spectrum in $K$ exists, then we may find such a UCP map which also maps $*$-polynomials in $N_1, \ldots, N_d$ to compact operators. Because the proof of Arveson's extension theorem relies on limits in a pointwise topology, not the norm topology, this may be too much to ask.

Below we consider a different sense of matrix approximation. Equip the matrix universe $\Md = \bigcup\limits_{n=1}^\infty M_n^d$ with a norm $||\cdot||$ that is decreasing under UCP maps. More precisely, equip each matrix level with a norm $||\cdot||_n$, such that if $\phi: M_n^d \to M_m^d$ is a UCP map and $A \in M_n^d$, then $||\phi(A)||_m \leq ||A||_n$. For example, we may take $||(A_1, \ldots, A_d)||$ to be the sum of the operator norms of each $A_i$. With the norm $||\cdot||$ fixed, we may consider the Hausdorff topology on subsets $\mathcal{S} = \bigcup_{n=1}^\infty \mathcal{S}_n$, which we abbreviate as follows.
\begin{definition}
We write $\mathcal{S} \approx_\varepsilon \mathcal{T}$ if for every $S \in \mathcal{S}$ there exists $T  \in \mathcal{T}$ with $||S - T|| < \varepsilon$, and for every $T^\prime \in \mathcal{T}$, there exists $S^\prime \in \mathcal{S}$ with $||T^\prime - S^\prime|| < \varepsilon$. We apply the same notation for subsets of a fixed matrix level $M_n^d$, so that $\mathcal{S} \approx_{\varepsilon} \mathcal{T}$ if and only if $\mathcal{S}_n \approx_\varepsilon \mathcal{T}_n$ for each $n$.
\end{definition}

\begin{proposition}\label{prop:approxWmin}
Let $A \in M_n^d$ have $\Wmin{}(K) \approx_\varepsilon \mathcal{W}(A)$. Then there is a polyhedron $L \subseteq K$ with at most $2n^3(d+1) + 1$ vertices such that $K \approx_{2\varepsilon} L$.
\end{proposition}
\begin{proof}
By definition, we have that for any $m \in \bZ^+$ and any tuple $T \in \Wmin{m}(K)$, there exists a UCP map $\phi: M_n^d \to M_m^d$ such that $||\phi(A) - T|| < \varepsilon$. Similarly, since $A \in \mathcal{W}_n(A)$, there exists an $n \times n$ matrix tuple $B \in \Wmin{n}(K)$ such that $||A - B|| < \varepsilon$, which implies that $||\phi(B) - T|| < 2 \varepsilon$. That is, any tuple in $\Wmin{}(K)$ may be approximated within $2 \varepsilon$ by a tuple in $\mathcal{W}(B)$. Since we also have that $\mathcal{W}(B) \subseteq \Wmin{}(K)$, it holds that $\mathcal{W}(B) \approx_{2\varepsilon} \Wmin{}(K)$.

Let $N$ be a normal dilation of $B$ with joint spectrum in $K$, where we may suppose the members of $N$ are matrices of dimension at most $2n^3(d+1) + 1$ by Proposition \ref{prop:introcara}. Then since $\mathcal{W}(B) \approx_{2\varepsilon} \Wmin{}(K)$ and $\mathcal{W}(B) \subseteq \mathcal{W}(N) \subseteq \Wmin{}(K)$, it holds that $\mathcal{W}(N) \approx_{2\varepsilon} \Wmin{}(K)$. Restricting to the first level yields $\mathcal{W}_1(N) \subseteq K$ and $\mathcal{W}_1(N) \approx_{2\varepsilon} K$. By normality of $N$, $L := \mathcal{W}_1(N)$ is a polyhedron with at most $2n^3(d+1) + 1$ vertices.
\end{proof}

When approximation is replaced by equality, the vertex count $2n^3(d+1) + 1$ may be replaced by the more pleasant $n$. This may be deduced from \cite{arvesonmatrix}, but we will present a proof which arises in pursuit of the following problem: if $T \in \cK(H)^d$ is a tuple of compact operators with $\mathcal{W}(T) = \Wmin{}(K)$, to what extent is the shape of $K$ restricted by the fact that $T$ is compact? In addition, if $K$ is fixed, to what extent is $T$ determined by $K$? Such results appeared to be within the scope of \cite[\S 6]{DDSS}, but one of our contributions here is the presentation of counterexamples to the claims therein. In response, the authors uploaded an arxiv correction \cite{DDSScorrected} which addresses these examples with additional assumptions. In both versions, the arguments center around a minimality condition for operator tuples, as in \cite[Definition 6.1]{DDSS}. We repeat that definition here.

\begin{definition}\label{def:minimalitydefW}
A tuple $T = (T_1, \ldots, T_d) \in \cB(H)^d$ is said to be \textit{minimal}, or \textit{minimal for its matrix range}, if the restriction of $T$ to any proper reducing subspace has strictly smaller matrix range.
\end{definition}

If one intends to determine $T$ uniquely from its matrx range $\mathcal{W}(T)$, the presence of some minimality condition (though not necessarily the one above) is natural. We show that even in the compact case, this particular condition is not sufficient to determine $T$ from $\mathcal{W}(T)$ up to unitary equivalence. Similarly, given a compact tuple $T$ which is not minimal, there might not be a summand which is minimal for the same matrix range. Finally, a minimal compact tuple need not have the property that the unital $C^*$-algebra it generates is isomorphic to the $C^*$-envelope of the operator system it generates. All three results were claimed positively in \cite{DDSS}, extending the uniqueness results of \cite{arvesonmatrix, relaxation, zalar} for matrix ranges or free spectrahedra to the compact setting. Therefore, our counterexamples show that Definition \ref{def:minimalitydefW} is insufficient for consideration of compact tuples acting on infinite-dimensional spaces. The reader is invited to read \cite[\S 6]{DDSScorrected} and see how the definition of \lq\lq nonsingularity\rq\rq\hspace{0pt} therein covers the non-pathological cases presented in this section (which motivated said definition).

Recall that for a matrix convex set $\mathcal{S}$, there are multiple relevant notions of \textit{extreme point}. Given a matrix convex combination
\be\label{eq:matrixconvexcombo}
X = \sum_{i=1}^m V_i^* Y^i V_i, \hspace{.7 in} Y^i \in \mathcal{S}_{n_i}, \hspace{.2 in} V_i: \bC^{n} \to \bC^{n_i}, \hspace{.2 in} \sum_{i=1}^m V_i^*V_i = 1,
\ee
one calls the point $X \in \mathcal{S}_n$
\begin{itemize}
\item an \textit{absolute extreme point} of $\mathcal{S}$ if whenever each $V_i: \bC^n \to \bC^{n_i}$ is nonzero, it follows that each $Y^i$ contains a summand (possibly equal to $Y^i$) which is unitarily equivalent to $X$.
\item a \textit{matrix extreme point} of $\mathcal{S}$ if whenever each $V_i: \bC^n \to \bC^{n_i}$ is surjective, it follows that each $Y^i$ is unitarily equivalent to $X$.
\item a \textit{Euclidean extreme point} of $\mathcal{S}$ if $X$ is an extreme point of the convex set $\mathcal{S}_n$ in the usual sense. 
\end{itemize}
These definitions may be characterized in dilation-theoretic terms, as in \cite[Theorem 1.1]{evertheltonext}. Note that at the scalar level $n = 1$, there is no difference between matrix extreme points and Euclidean extreme points, as surjectivity of the $V_i$ forces $n_i = 1$, in which case (\ref{eq:matrixconvexcombo}) reduces to a traditional convex combination. 

We will need the following lemma, which concerns the application of vector states to normal tuples.

\begin{lemma}\label{lem:AEPmin}
Let $K \subset \bC^d$ be a compact convex set and suppose $\lambda$ is an extreme point of $K$. If $N \in \cB(\cH)^d$ is a normal tuple with $\sigma(N) \subseteq K$, and some $v \in \cH$ has $||v|| = 1$ and $\langle N_j v, v \rangle = \lambda_j$ for each $j$, then $N_j v = \lambda_j v$ for each $j$.
\end{lemma}
\begin{proof}
The joint reducing subspace of $N$ generated by $v$ is separable, so we may assume $\cH$ is separable. 
Therefore $\cH$ is a finite or countable direct sum of $L^2(\mu_i)$ for regular Borel measures $\mu_i$ on $\sigma(N)$, which we may extend to measures on $K$ in the usual way. We may write $N_j = \oplus M_{\pi_j}$ for $\pi_j$ the $j$th coordinate function on $\sigma(N)$, and $v = (f_1, f_2, \ldots )$ for $f_i \in L^2(\mu_i)$ with $\sum \int |f_i|^2 \,d\mu_i = 1$. We need to prove that $\pi_j f_i = \lambda_j f_i \text{ a.e. } [\mu_i]$. Equivalently, if we let $d\nu_i = |f_i|^2 \,d\mu_i$ and $\nu(E) = \sum \nu_i (E)$, so $\nu$ is a probability measure, we need to prove that $\pi_j = \lambda_j \text{ a.e. } [\nu]$. 

By definition, it holds that for each $j \in \{1, \ldots, d\}$,
\bes
\int_K \pi_j \,d\nu = \sum \int_K \pi_j \,d\nu_i = \sum \int_K \pi_j \cdot |f_i^2| \,d\mu_i = \sum \langle M_{\pi_j} f_i, f_i \rangle = \langle N_j f, f \rangle = \lambda_j = \pi_j(\lambda).
\ees
Therefore, for the affine function system $S = \textrm{span}\{1, \pi_1, \ldots, \pi_d\} \subset C(K)$, $\nu$ is a representing measure for the point $\lambda$. Since $\lambda$ is an extreme point of $K$, it is a Choquet boundary point of $S$, and the representing measure is unique: $\nu = \delta_\lambda$. Combined with the above equality $\int \pi_j \,d\nu = \lambda_j$, this says that $\pi_j = \lambda_j \text{ a.e. } [\nu]$. It follows that $\pi_j f_i = \lambda_j f_i \text{ a.e. } [\mu_i]$ for each $i$ and $j$, and the proof is complete.
\end{proof}

In Lemma \ref{lem:AEPmin}, the tuple $N$ may act on an infinite-dimensional space. Even so, some manipulation of matrix ranges shows that Lemma \ref{lem:AEPmin} is actually equivalent to the claim
\be\label{eq:AEPEP}
\lambda \text{ is an extreme point of } K \implies \lambda \text{ is an absolute extreme point of } \Wmin{}(K).
\ee
We leave the details of the equivalence between Lemma \ref{lem:AEPmin} and (\ref{eq:AEPEP}) to the reader, and we remark that (\ref{eq:AEPEP}) is certainly well-known. In particular, any minimal set $\Wmin{}(K)$ is spanned by its absolute extreme points. See also the general result \cite[Corollary 6.12]{Kriel}, which shows that if  $\mathcal{S}$ is a matrix convex set spanned by its matrix extreme points from a fixed level $\mathcal{S}_n$, then $\mathcal{S}$ is also spanned by its absolute extreme points. That result pairs quite nicely with \cite[Corollary 1.1]{evertabs}, which concerns precisely the opposite scenario -- there exists a family of matrix convex sets which have no absolute extreme points at all. The construction of this family relies heavily on the use of compact operators acting on infinite-dimensional spaces.

We will use Lemma \ref{lem:AEPmin} in tandem with the following facts, which may be demonstrated by matrix computations. 

\begin{itemize}
\item If $T$ is a $d$-tuple of matrices and $\lambda$ is an extreme point of $\mathcal{W}_1(T)$, then there is a unit vector $v$ with $\langle T_i v, v \rangle = \lambda_i$.

\item If $T$ is a $d$-tuple of compact operators and $\lambda$ is a {\blue \textbf{nonzero}} extreme point of $\mathcal{W}_1(T)$, then there is a unit vector $v$ with $\langle T_i v, v \rangle = \lambda_i$.
\end{itemize}
In particular, we note that detection of the point $0 \in W_1(T)$ for $T$ a compact tuple on an infinite-dimensional space might not be achieved using a vector state. For example, consider the single diagonal operator $S = \bigoplus\limits_{n \in \bZ^+} \frac{1}{n}$.

If $T \in \cK(H_1)^d$ has $\mathcal{W}(T) = \Wmin{}(K)$, then $T$ dilates to a normal tuple $N \in \cB(H_2)^d$ with the same matrix range by Proposition \ref{prop:introcara}, though compactness of $T$ might be lost in the dilation. If $v$ is a unit vector in $H_1$, then by definition, the corresponding vector state gives the same result whether it is applied to $T$ or to $N$. Further, if $v$ is a joint eigenvector for $N$, then because $v$ belongs to the smaller Hilbert space $\cH_1$, it is also a joint eigenvector for $T$. We may use these facts to characterize when $\mathcal{W}(T) = \Wmin{}(K)$ for $T$ a compact or matrix tuple. While we find that the shape of $K$ is what one would expect if $T$ were actually normal, $T$ itself might be minimal and non-normal.

\begin{theorem}\label{thm:matrix_eigenvector_hunt}
Let $T \in M_n^d$ be a tuple of  matrices with $\mathcal{W}(T) = \Wmin{}(K)$. Then every extreme point of $K$ is a joint eigenvalue of $T$, and in particular $K$ must be a polyhedron with at most $n$ vertices.
\end{theorem}
\begin{proof}
Since $K = \mathcal{W}_1(T)$ and $T$ is finite-dimensional, if $\lambda$ is an extreme point of $K$, there is a vector state $v$ such that $\langle T_i v, v \rangle = \lambda_i$. Writing a normal dilation $N$ of $T$, $\textrm{conv}(\sigma(N)) = K$, we have that similarly $\langle N_i v, v \rangle = \lambda_i$. By Lemma \ref{lem:AEPmin}, since $\lambda$ is extreme and detected by the vector state $v$, we conclude that $N_i v = \lambda_i v$. Now, as $v$ belongs to the Hilbert space on which $T$ acts, we also have $T_i v = \lambda_i v$, i.e. $v$ is a joint eigenvector for $T$. Finally, as $T$ consists of $n \times n$ matrices, there are only up to $n$ possible extreme points of $K$.
\end{proof}

Theorem \ref{thm:matrix_eigenvector_hunt} and its proof are listed for completeness, as whenever $T$ is a matrix tuple, we may assume $T$ is of the smallest dimension possible and apply any one of the various uniqueness results for free spectrahedra or matrix ranges of matrices. See, for example, \cite[\S 1]{arvesonmatrix}, \cite[Theorem 3.12 and Proposition 3.17]{relaxation}, \cite[Theorem 1.2]{zalar}, and \cite[Definition 6.3 and Theorem 6.9]{DDSScorrected}. We now consider a version for compact operator tuples.

\begin{theorem}\label{thm:compact_eigenvector_hunt}
Let $H$ be a Hilbert space (of any finite or infinite dimension). If there exists a tuple $T \in \cK(\cH)^d$ of compact operators with $\mathcal{W}(T) = \Wmin{}(K)$, then every {\blue \textbf{nonzero}} extreme point of $K$ is a joint eigenvalue for $T$, and $\text{ext}(K)$ is either a finite set or a sequence tending to zero.
\end{theorem}
\begin{proof}
We begin by following the logic of the previous proof, noting that since $T$ is a compact tuple, we may detect {\blue \textbf{nonzero}} extreme points $\lambda$ of $K$ through vector states, $\lambda_i = \langle T_i v, v \rangle$. We similarly conclude that each such $\lambda$ is a joint eigenvector of $T$ using a normal dilation, Lemma \ref{lem:AEPmin}, and a restriction.
Since $\text{ext}(K) \setminus \{0\}$ is contained in the set of eigenvalues of a compact operator tuple, it follows that $\text{ext}(K)$ is either a finite set or a sequence tending to zero.
\end{proof}

Theorems \ref{thm:matrix_eigenvector_hunt} and \ref{thm:compact_eigenvector_hunt} show that if a matrix tuple, or a compact tuple acting on an infinite-dimensional space, has $\mathcal{W}(T) = \Wmin{}(K)$, then the shape of $K$ is \lq\lq precisely what one would expect\rq\rq\hspace{0pt} from examination of normal tuples with the same properties. In particular, the theorems exhibit a decomposition $T \cong N \oplus M$ where $N$ is a normal tuple. However, we note that the qualification of {\blue \textbf{nonzero}} extreme point in Theorem \ref{thm:compact_eigenvector_hunt} is problematic, in that if $K$ is a polyhedron with $0$ as a vertex, we might have that $\mathcal{W}(N)$ is a proper subset of $\Wmin{}(K)$. In particular, we might not be able to find a minimal normal summand for the same matrix range.

There are two distinct questions one can consider regarding minimality in this context, as in \cite{DDSS}. First, if a compact tuple $T$ has matrix range $\Wmin{}(K)$, and $T$ is minimal for its matrix range, is $T$ determined up to unitary equivalence? Second, if $T$ is not minimal, does it have a summand which is minimal for the same matrix range? Theorems \ref{thm:matrix_eigenvector_hunt} and \ref{thm:compact_eigenvector_hunt} answer both questions affirmatively in most cases, but there are pathological examples concerning the point $0 \in \mathcal{W}_1(T)$ when $T$ acts on an infinite-dimensional space. First, we consider the question of uniqueness.

\begin{corollary}\label{cor:good_uniqueness}
Let $K \subseteq \bC^d$ be a compact convex set, and let $H$ be a Hilbert space (of any finite or infinite dimension). If $T \in \cK(H)^d$ is minimal for its matrix range $\mathcal{W}(T) = \Wmin{}(K)$, then the following hold.
\begin{enumerate}
\item If $K$ has infinitely many extreme points, then $\cH$ is separable and infinite-dimensional, and $T$ is diagonal with eigenvalues at the nonzero extreme points of $K$, which are isolated in the extreme points and form a sequence tending to zero.
\item If $T$ acts on an $n$-dimensional space $H$, then $K$ is a polyhedron with exactly $n$ vertices, and $T$ is diagonal with eigenvalues at the vertices of $K$.
\item If $K$ is a polyhedron with $n$ vertices, {\blue \textbf{none of which are 0}}, then $\cH$ must be $n$-dimensional, and $T$ is diagonal with eigenvalues at the vertices of $K$.
\end{enumerate}
\end{corollary}
\begin{proof}
In cases (1) and (3), application of Theorem \ref{thm:compact_eigenvector_hunt} shows that $T$ admits eigenvectors for joint eigenvalues at each nonzero extreme point of $K$. The assumptions of either case show that the resulting normal summand $N$ of $T$ has $\mathcal{W}(N) = \Wmin{}(K)$. By minimality of $T$, we have that $N = T$, so $T$ is diagonal with the prescribed eigenvalues and the dimension of $H$ is determined by the number of nonzero extreme points of $K$, which is either finite or a sequence tending to zero from Theorem \ref{thm:compact_eigenvector_hunt}. In case (2), we instead apply Theorem \ref{thm:matrix_eigenvector_hunt}, so $T$ has joint eigenvectors for each vertex of $K$ (possibly including zero). By minimality, $T$ must be equal to the resulting normal summand, so $T$ is diagonal and the dimension of $H$ is determined by the number of vertices $n$ of $K$.
\end{proof}

We note that if $H$ is assumed finite-dimensional, then more general uniqueness results were proved in \cite{arvesonmatrix, relaxation, zalar}, either in terms of matrix ranges or free spectrahedra. Therefore, Corollary \ref{cor:good_uniqueness} is primarily of use to determine the shape of $K$ based on $T$ (or vice-versa) when $H$ is infinite-dimensional, or to determine when the dimension of $H$ must be finite or infinite based on other assumptions.

In all cases of the \lq\lq non-pathological\rq\rq\hspace{0pt} Corollary \ref{cor:good_uniqueness}, $T$ is normal, and in particular, $T$ is diagonal with eigenvalues at the isolated extreme points of $K$. Later, in Theorem \ref{thm:easy_spec_thm}, we will demonstrate that regardless of compactness, all minimal \textit{normal} tuples will take this form. In particular, if a normal tuple $N$ is minimal for matrix range $\mathcal{W}(N) = \Wmin{}(K)$, then the isolated extreme points of $K$ must be dense in $\text{ext}(K)$. However, not all minimal tuples are normal, even in the compact case. Indeed, the following example shows why item (3) of Corollary \ref{cor:good_uniqueness} must be only a partial converse to item (2).

\begin{example}\label{exam:simplex_surprise}
Let $K = \textrm{span}\{ (0,0), (1,0), (0, 1) \}$ be the standard simplex in $\bR^2$. Given an orthonormal basis $\{e_1, e_2, \ldots \}$ of an infinite-dimensional separable Hilbert space, let $v_1 = \cfrac{1}{\sqrt{2}} \hspace{3pt} e_1 + \cfrac{1}{\sqrt{4}}  \hspace{3pt} e_2 + \cfrac{1}{\sqrt{8}}  \hspace{3pt} e_3 + \ldots$, and extend $v_1$ to an orthonormal basis $\{v_1, v_2, \ldots\}$ for the same space.

Define the operators $S_i$ and $T_i$ as follows.
\bes
S_1 = \sum_{n=1}^\infty \frac{1}{3n} P_{e_n} \hspace{.3 in} S_2 = \sum_{n=1}^\infty \frac{1}{3n} P_{v_n} \hspace{.3 in} T_1 = \begin{pmatrix} 1 \\  & 0  \\ & &  S_1 \end{pmatrix} \hspace{.3 in} T_2 = \begin{pmatrix} 0 \\  & 1  \\ & &  S_2 \end{pmatrix} 
\ees
Now, $T = (T_1, T_2) \in \cK(H)^2_{sa}$ is a compact tuple of positive operators whose numerical range includes $(0,0)$, $(1,0)$, and $(0, 1)$, where we note that $(0,0)$ is detected by a \textit{limit} of vector states. Moreover, we have $0 \leq T_1 + T_2 \leq I$, so it follows that $W_1(T)$ is precisely equal to $K$. Since $K$ is a simplex, there is only one matrix convex set with $K$ as its scalar level, and $\mathcal{W}(T) = \Wmin{}(K)$. 

Suppose $\widetilde{H}$ is a reducing subspace for $T$ such that the restriction $R$ has $\mathcal{W}(R) = \Wmin{}(K)$. By Theorem \ref{thm:compact_eigenvector_hunt}, $R$ has joint eigenvectors for the eigenvalues $(1,0)$ and $(0,1)$. Since $S_1 + S_2 \leq \frac{2}{3} I$, the only possible eigenvectors are those exhibited in the direct sum decomposition of $T$. Therefore, to show $\widetilde{H} = H$, we need only show that $\widetilde{H}$ includes the entire domain of the summand $S = (S_1, S_2)$. The intersection of $\widetilde{H}$ to the domain of $S$ is a reducing subspace for $S$, which we denote by $L$. The subspace $L$ is nontrivial, as we must have $(0,0) \in W_1(R)$, so $L$ includes a vector $x$ with $\langle x, e_n \rangle \not= 0$ for some fixed $n$. Since $S_1$ has distinct nonzero eigenvalues at the $e_i$, manipulation of the functional calculus shows that this particular $e_n$ belongs to $L$. However, $\langle e_n, v_1 \rangle \not= 0$, so applying the same trick to the eigenvector basis of $S_2$ shows that $v_1 \in L$. Finally, for all $m \in \bZ^+$, $\langle v_1, e_m \rangle \not= 0$, so examining $S_1$ again shows that $e_m \in L$ for all $m$, and $L$ is the entire domain of $S$. That is, $S$ is irreducible, $\widetilde{H} = H$, and $T$ is minimal for its matrix range. 

Finally, note that by replacing $\frac{1}{3n}$ with $\frac{1}{3n^p}$, $p \in [1, \infty)$, one can construct uncountably many examples of $T$, no two of which are unitarily equivalent.
\end{example}

\begin{corollary}\label{cor:uniqueness_issues}
Let $K \subseteq \bC^d$ be a polyhedron with at least 3 vertices, one of which is $0$, and fix a separable infinite-dimensional Hilbert space $H$. Then there are uncountably many unitarily inequivalent tuples $T \in \cK(H)^d$ such that $\mathcal{W}(T) = \Wmin{}(K)$ and $T$ is minimal for its matrix range.
\end{corollary}
\begin{proof}
Let $T$ have joint eigenvalues at all of the nonzero vertices, and produce an additional summand $S$ of $T$ which is determined from a selection of three vertices $0, v_1, v_2$ and the technique of Example \ref{exam:simplex_surprise}, after an invertible linear transformation.
\end{proof}

The above results may also be formulated in the language of operator systems, after we recall some additional notation and definitions.

\begin{definition}
If $T = (T_1, \ldots, T_d) \in \cB(H)^d$, then we let $S_T$ denote the operator system (i.e., self-adjoint unital subspace of $\cB(H)$) generated by $T_1, \ldots, T_d$. Given an operator system $S$, we let $C^*(S)$ denote the $C^*$-algebra generated by $S$, which is necessarily unital.
\end{definition}

If $S$ is an operator system inside a unital $C^*$-algebra $\mathcal{A}$, the particular structure of $\mathcal{A}$ is not generally relevant unless $S$ generates $\mathcal{A}$. It is usually easier to ignore $\mathcal{A}$ and write \lq\lq $S \subseteq C^*(S)$ is an operator system,\rq\rq\hspace{0pt} where it is important to note that the operator system structure of $S$ alone might not determine the $C^*$-algebra $S$ generates. That is, it is possible for two operator systems $S_1$ and $S_2$ to be completely isometrically isomorphic while $C^*(S_1)$ and $C^*(S_2)$ are not isomorphic. However, given any concrete representation $S \subseteq C^*(S)$, there is a quotient of $C^*(S)$ that produces the \lq\lq smallest\rq\rq\hspace{0pt} $C^*$-algebra into which $S$ embeds.

\begin{definition}
Let $S$ be an operator system. The \textit{Shilov ideal} of $S$ inside $C^*(S)$ is the largest ideal $I$ such that the quotient $C^*(S) \to C^*(S)/I$ is completely isometric on $S$. The $C^*$-\textit{envelope} of $S$, denoted $C^*_e(S)$, is the quotient of $C^*(S)$ by the Shilov ideal.
\end{definition}

The existence of the Shilov ideal is a very deep result, and the structure of $C^*_e(S)$ does not depend on the choice of the initial concrete representation of $S$ inside a $C^*$-algebra (see \cite{evgenios} for additional background). We begin with a natural and well-known example.

\begin{example}\label{ex:naturalenvelope}
Let $T \in \cB(H)^d$ be a tuple with $\mathcal{W}(T) = \Wmin{}(K)$. Then $C^*_e(S_T)$ is isomorphic to the commutative $C^*$-algebra $C(\overline{\text{ext}(K)})$ of continuous complex-valued functions on $\overline{\text{ext}(K)}$. To see this, let $Z = (Z_1, \ldots, Z_d)$ denote the tuple of coordinate functions on $\overline{\text{ext}(K)}$. Since the convex hull of $\overline{\text{ext}(K)}$ is $K$, we have that $\mathcal{W}(Z) = \Wmin{}(K) = \mathcal{W}(T)$, and there is a unital completely isometric map $S_T \to S_Z$ mapping $T_i \mapsto Z_i$ by \cite[Theorem 5.1]{DDSS}. It follows that $C_e^*(S_T) \cong C_e^*(S_Z)$ is a quotient of $C^*(S_Z) = C(\overline{\text{ext}(K)})$. If $I$ is a nontrivial ideal in $C^*(S_Z)$, then $C^*(S_Z) / I$ may be written as $C(X)$ for a proper compact subset $K$ of $\overline{\text{ext}(K)}$. If $\widetilde{Z} := (Z_1 + I, \ldots, Z_d + I)$, then it follows that $\mathcal{W}(\widetilde{Z}) = \Wmin{}(\text{conv}(X))$, where $\text{conv}(X)$ is a proper subset of $K$. In particular, $\mathcal{W}(\widetilde{Z}) \neq \mathcal{W}(Z)$, so the quotient map by $I$ is not completely isometric as a map $S_Z \to S_{\widetilde{Z}}$ by \cite[Theorem 5.1]{DDSS}. We conclude that the Shilov ideal of $S_Z$ in $C^*(S_Z)$ is trivial, and $C_e^*(S_T) \cong C_e^*(S_Z) \cong C(\overline{\text{ext}(K)})$.
\end{example}

Because the $C^*$-envelope computed above is certainly commutative, we find that each example found using the construction in Corollary \ref{cor:uniqueness_issues} must have nontrivial Shilov ideal.

\begin{corollary}
For any $d \geq 2$, there exist uncontably many unitarily inequivalent tuples $T \in \cK(H)^d_{sa}$ of compact self-adjoint operators such that $T$ is minimal for its matrix range, but $C^*(S_T) \not\cong C^*_e(S_T)$. In particular, the Shilov ideal of $S_T$ in $C^*(S_T)$ is nontrivial.
\end{corollary}
\begin{proof}
In Corollary \ref{cor:uniqueness_issues}, the constructed tuple $T$ has $C^*_e(S_T)$ isomorphic to the commutative $C^*$-algebra of functions on the finite set $\text{ext}(K)$. However, the operators $T_i$ do not commute, so $C^*(S_T)$ cannot be isomorphic to $C^*_e(S_T)$, and the Shilov ideal must be nontrivial.
\end{proof}

Since there are now numerous examples of minimal compact tuples which are not uniquely determined by the matrix ranges, even when the matrix ranges considered are of the form $\Wmin{}(K)$, we can consider relaxing the problem somewhat. One option is to replace minimality with a stronger condition, as in the following definition.

\begin{definition}\label{def:fullycompressedtuples}
A tuple $T = (T_1, \ldots, T_d) \in \cB(H)^d$ is said to be \textit{fully compressed} if the compression of $T$ to any proper subspace of $H$ has strictly smaller matrix range.
\end{definition}

Any tuple which is fully compressed is also minimal in the sense of Definition \ref{def:minimalitydefW}, but the reverse implication certainly does not hold. In fact, it is easy to see that if a compact tuple has matrix range $\mathcal{W}(T) = \Wmin{}(K)$, then the assumption that $T$ is fully compressed allows us to uniquely determine $T$. That is, the counterexamples we consider in this section are no longer counterexamples in the new framework.

\begin{proposition}\label{prop:compression_min_Ik}
Suppose that $S \in \cK(H_1)^d$ and $T \in \cK(H_2)^d$ are fully compressed compact tuples with matrix range $\Wmin{}(K)$. Then $S$ and $T$ are unitarily equivalent. 
\end{proposition}
\begin{proof}
Since any fully compressed tuple is automatically minimal, Corollary \ref{cor:good_uniqueness} shows that $S$ and $T$ are unitarily equivalent if $\text{ext}(K)$ is an infinite set, or if $K$ is a polyhedron which does not have $0$ as a vertex. If $K$ is a polyhedron with $0$ as a vertex, then any fully compressed compact tuple with matrix range $\Wmin{}(K)$ is of the form $\lambda_1 \oplus \cdots \oplus \lambda_n \oplus Q$, where $\lambda_1, \ldots, \lambda_n$ are the nonzero vertices of $K$ and $0 \in \mathcal{W}_1(Q)$. If $Q$ is anything other than $(0, \ldots, 0) \in \bC^d$, there is a proper compression $Q^\prime$ of $Q$ with $0 \in \mathcal{W}_1(Q)$, and hence a proper compression of the original tuple with the same matrix range $\Wmin{}(K)$.
\end{proof}

It is conceivable that in the general setting, if $T$ is a fully compressed tuple, then $\mathcal{W}(T)$ determines $T$ up to unitary equivalence. However, one should expect the closed and bounded matrix convex sets $\mathcal{S}$ which may be obtained as the matrix ranges of fully compressed tuples to be fairly restricted. 

We now consider a separate, but very much related problem which was examined in \cite{DDSS}. If $T$ is compact with $\mathcal{W}(T) = \Wmin{}(K)$, does there exist a summand of $T$ which is minimal for the same matrix range? For this problem, there is again an issue with the point $0$ and detection by non-vector states, but the obstruction is far more elementary.

\begin{example}\label{ex:not_a_surprise}
Let $K \not= \{0\}$ be a polyhedron with vertices $0, v_1, \ldots, v_n$. Then the diagonal tuple $N$ with joint eigenvalues $v_1, \ldots, v_n, \frac{1}{2}v_n, \frac{1}{3}v_n, \ldots$ is compact and normal with matrix range $\mathcal{W}(N) = \Wmin{}(K)$. However, it is not minimal for its matrix range, and it admits no minimal summand with the same matrix range, as any summand with the same matrix range has infinitely many eigenvalues converging to zero. 
\end{example} 

However, if one carefully avoids zero, then minimal summands are easy to pick out, as in the following corollary.

\begin{corollary}\label{cor:good_summands}
Let $T \in \cK(H)^d$ have matrix range $\mathcal{W}(T) = \Wmin{}(K)$. Then in any of the following circumstances, there exists a decomposition $T \cong N \oplus M$ where $N$ is minimal for the same matrix range. 
\end{corollary}
\begin{enumerate}
\item $H$ has finite dimension (so $K$ is a polyhedron with at most $\text{dim}(H)$ vertices), or
\item $\textrm{ext}(K)$ is infinite (so $\text{dim}(H)$ is infinite, and $\textrm{ext}(K)$ is a sequence tending to $0$), or
\item $K$ is a polyhedron and $0$ is not a vertex of $K$, or
\item $T$ has $0$ as a joint eigenvalue.
\end{enumerate}
We may choose $N$ diagonal with eigenvalues at isolated extreme points of $K$. For cases (1), (2), and (3), this is the only choice of $N$. For case (4), the choice might not be unique.
\begin{proof}
For case (1), application of Theorem \ref{thm:matrix_eigenvector_hunt} shows that $K$ is a polyhedron with at most $\text{dim}(H)$ vertices and $T \cong N \oplus M$, where $N$ is diagonal with eigenvalues for each vertex.

For case (2), Theorem \ref{thm:compact_eigenvector_hunt} shows that $\textrm{ext}(K)$ is a sequence tending to $0$, and we may write $T \cong N \oplus M$ where $N$ is diagonal with a single eigenvalue for each nonzero extreme point of $K$. Note in particular that in this case, every point of $\text{ext}(K) \setminus \{0\}$ is isolated, and $0$ is either not an extreme point, or it is an extreme point which is not isolated. In either case, $0$ is not used as an eigenvalue of $N$, as it is not needed.

For case (3), Theorem \ref{thm:compact_eigenvector_hunt} produces eigenvectors for each vertex of $K$, so $T \cong N \oplus M$ for $N$ the diagonal operator with those eigenvalues.

For case (4), we need only consider a polyhedron $K$ with $0$ as a vertex, as otherwise we may apply case (2) or (3).  Since $T$ by assumption has a joint eigenvector for $0$, and Theorem \ref{thm:compact_eigenvector_hunt} produces joint eigenvectors for the other vertices of $K$, we find that $T \cong N \oplus M$ where $N$ is a finite-dimensional diagonal tuple with an eigenvalue corresponding to each vertex of $K$.

Finally, we consider uniqueness. In cases (1), (2), and (3), if $M$ is another minimal summand of $T$ with the same matrix range, then Theorems \ref{thm:matrix_eigenvector_hunt} and \ref{thm:compact_eigenvector_hunt} produce the diagonal operator $N$ as a summand of $M$, a contradiction. In case (4), however, it is possible that $K$ is a polyhedron with vertices $0$, $v_1$, $v_2$, \ldots, $v_n$, and that $T$ is the direct sum of $0$ and an operator $S$ of the form in Corollary \ref{cor:uniqueness_issues}, which has each $v_i$ as a joint eigenvalue as well as an additional irreducible summand. We may choose a minimal summand $N$ which is diagonal with eigenvalues at $0, v_1, \ldots, v_n$, or we may choose a minimal summand $M = S$.
\end{proof}

As in the comments after Theorem \ref{thm:matrix_eigenvector_hunt}, we note that there is some overlap with Corollary \ref{cor:good_summands} and results in which uniqueness results for matrix ranges or their polar duals apply. We also note that the pathology involving $0$ disappears if one accepts summands of a tuple that is approximately unitarily equivalent to $T$.

\begin{corollary}\label{prop:approxunitaryequivalencefix}
Suppose $T \in \cK(H)^d$ is a tuple of compact operators with matrix range $\mathcal{W}(T) = \Wmin{}(K)$. Then there is a decomposition $H \cong \widetilde{H} \oplus X$ and a tuple $T^\prime \in \cK(H)^d$ such that $T^\prime$ is approximately unitarily equivalent to $T$, $T^\prime$ decomposes as $T^\prime = S \oplus X$, and $S \in \cB(\widetilde{H})^d$ is minimal for the matrix range $\mathcal{W}(S) = \Wmin{}(K)$.
\end{corollary}
\begin{proof}
We may assume that $H$ is infinite-dimensional. From Corollary \ref{cor:good_summands}, a minimal summand of $T$ exists if $K$ is a polyhedron with $0 \not\in \text{ext}(K)$, or if $\text{ext}(K)$ is an infinite sequence tending to $0$. The only remaining case is that $K$ is a polyhedron which has $0$ as a vertex. In this case, $T$ decomposes as $\lambda_1 \oplus \cdots \lambda_n \oplus Q$, where the $\lambda_i$ are the nonzero vertices of $K$ and $Q$ is some infinite-dimensional compact tuple (which may or may not have $0$ as a joint eigenvalue). Since $Q$ is approximately unitarily equivalent to $0 \oplus Q$, an operator $T^\prime$ approximately unitarily equivalent to $T$ admits a finite-dimensional normal summand $S$ with joint eigenvectors at every vertex of $K$. This summand $S$ is therefore minimal.
\end{proof}

We now consider operator tuples which are not necessarily compact. A simple spectral theorem argument will show that minimal \textit{normal} tuples for matrix range $\Wmin{}(K)$ exist if and only if $K$ satisfies a geometric condition on its extreme points: the isolated extreme points of $K$ are dense in $\text{ext}(K)$. The condition also implies uniqueness of minimal normal tuples with $\mathcal{W}(T) = \Wmin{}(K)$. However, the same exact condition also guarantees there are many \textit{non-normal} minimal tuples with the same matrix range. 

Given a compact convex set $K \subseteq \bC^d$, let
\bes
\mathcal{I}_K = \{x \in \text{ext}(K): \exists V \subset \bC^d \textrm{ open  such that } V \cap \text{ext}(K) = \{x\}\}
\ees
be the set of isolated extreme points, and note that we may replace either instance of $\text{ext}(K)$ (or both instances) with $\overline{\text{ext}(K)}$ without changing the result. Further, if $r > 0$ is fixed, then there are only finitely many $x \in \text{ext}(K)$ which satisfy $||x - y|| \geq r$ for all $y \in \text{ext}(K) \setminus \{x\}$, so $\mathcal{I}_k$ is a finite or countably infinite set. Moreover, a point $x \in K$ is in $\mathcal{I}_K$ if and only if the closed convex hull of $\overline{\textrm{ext}(K)} \setminus \{x\}$ is a proper subset of $K$.

\begin{lemma}\label{lem:min_extreme_spectrum}
Let $N \in \cB(H)^d$ be a normal tuple which is minimal for its matrix range $\mathcal{W}(N) = \Wmin{}(K)$. Then $\sigma(N) = \ol{\text{ext}(K)}$.
\end{lemma}
\begin{proof}
Since $\mathcal{W}(N) = \Wmin{}(K)$, we have that $\text{conv}(\sigma(N)) = K$ by Proposition \ref{prop:minWrange}. Note that a closure of the convex hull is not needed.

First, suppose that $\ol{\text{ext}(K)} \not\subseteq \sigma(N)$. Since $\sigma(N)$ is closed, it follows that there is an extreme point $x$ of $K$ which is missing from $\sigma(N)$. From $\textrm{conv}(\sigma(N)) = K$ we find that $x$ is nontrivial convex combination of points from $\sigma(N) \subseteq K$, a contradiction.

Next, assume that $\ol{\text{ext}(K)}$ is properly contained in $\sigma(N)$. Let $V$ be an open set in $\bC^d$ such that $V \cap \ol{\text{ext}(K)} = \varnothing$ and there exists $y \in \sigma(N) \cap V$. Letting $(\pi_1, \ldots, \pi_d)$ denote the coordinate functions on $\sigma(N)$, there is a spectral measure $E$ defined on Borel subsets of $\sigma(N)$ with
\bes
N_i = \int_{\sigma(N)} \pi_i \,\,dE,
\ees
and the support of $E$ is exactly $\sigma(N)$. Since $V$ is open and intersects $\sigma(N)$, it follows that $E(V \cap \sigma(N))$ is a nonzero projection, and the set of Borel functions which vanish on $V \cap \sigma(N)$ is a proper, nonzero reducing subspace of each $N_i$. The restriction to the complement, i.e.
\bes
L_i = \int_{\sigma(N)} \pi_i \cdot I_{\sigma(N) \setminus V} \,\,dE,
\ees
has joint spectrum which includes all of $\ol{\text{ext}(K)}$, since $V \cap \overline{\text{ext}(K)} = \varnothing$. It follows that $\text{conv}(\sigma(L)) = K$ and $\mathcal{W}(L) = \Wmin{}(K)$, a contradiction of the minimality of $N$.
\end{proof}

\begin{theorem}\label{thm:easy_spec_thm}
Let $K$ be a nonempty compact convex subset of $\bC^d$. Then there is a normal, minimal tuple $N$ with $\mathcal{W}(N) = \Wmin{}(K)$ if and only if $\ol{\mathcal{I}_K} = \ol{\text{ext}(K)}$. Moreover, in this case, the only such $N$ is diagonal with eigenvalues at each $\lambda \in \mathcal{I}_K$.
\end{theorem}
\begin{proof}

\textbf{Step I}. Assume $\ol{\mathcal{I}_K} = \ol{\text{ext}(K)}$. Let $D$ be diagonal with joint eigenvalues from the countable set $\mathcal{I}_K$, so $\sigma(D) = \ol{\mathcal{I}_K} = \ol{\text{ext}(K)}$ and $\mathcal{W}(D) = \Wmin{}(K)$. Any projection that commutes with $D$ corresponds to a direct sum of eigenspaces for $D$, so if $\widetilde{H}$ is a proper reducing subspace for $D$, then the restriction $R$ misses an eigenvalue $x \in \mathcal{I}_K$. Since $x$ is isolated, the convex hull of $\sigma(R)$ does not equal $K$, and therefore $\mathcal{W}(R) \not= \Wmin{}(K)$. We conclude that $D$ is minimal, and in particular a normal minimal tuple for matrix range $\Wmin{}(K)$ exists.
If $N$ is any other minimal normal tuple for the same matrix range, then by Lemma \ref{lem:min_extreme_spectrum}, $\sigma(N) = \overline{\text{ext}(K)}$, so every point $x \in \mathcal{I}_K$ is an atom of the spectral measure. It follows that $N$ has eigenvalues at each $x \in \mathcal{I}_K$, and $D$ is a summand of $N$ with the same matrix range. By minimality, $N$ must equal $D$.

\textbf{Step II}. Suppose there exists a minimal normal $N$, but $\ol{\mathcal{I}_K}$ is a proper subset of $\ol{\text{ext}(K)}$. We know from Lemma \ref{lem:min_extreme_spectrum} that $\sigma(N) = \ol{\text{ext}(K)}$. Let $O$ be an open set in $\bC^d$ such that $O \cap \text{ext}(K) \not= \varnothing$ and $O \cap \ol{\mathcal{I}_K} = \varnothing$, and let $E$ be the spectral measure that represents $N$. 

\textbf{Case (a)}. Suppose there is a point $x \in O \cap \text{ext}(K)$ with $E(\{x\}) \not= 0$. Since $x$ is not an isolated extreme point, it follows that $E(\ol{\text{ext}(K)} \setminus \{x\})$ is a projection onto a proper reducing subspace of $N$ whose restriction has the same spectrum, a contradiction of minimality.

\textbf{Case (b)}. Suppose that for every $x \in O \cap \text{ext}(K)$, $E(\{x\}) = 0$. Since the support of $E$ is $\sigma(N) = \ol{\text{ext}(K)}$, and $O$ is an open set which includes some extreme points, it follows that $E(O)$ must be a nontrivial projection. Let $z$ be a unit vector in its range, so $\langle E(O)z, z \rangle = 1$. Also let $x_n \in O \cap \sigma(N)$ form a sequence which is dense in $\ol{O} \cap \sigma(N)$. Because $E(\{x_n\}) = 0$, there is an open neighborhood $x_n \in V_n \subset O$ such that $\langle E(V_{n}) z, z \rangle \leq \frac{1}{2^{n+1}}$ holds, and
\[
\left \langle E \left( O^c \cup \bigcup_{n=1}^\infty V_n \right) z, z \right \rangle = \left\langle E \left( \bigcup_{n=1}^\infty V_n \right) z, z \right \rangle \leq \sum_{n=1}^\infty \langle E(V_n) z, z \rangle \leq \sum_{n=1}^\infty \frac{1}{2^{n+1}} = \frac{1}{2}.
\]
Therefore, $P := E \left( O^c \cup \bigcup\limits_{n=1}^\infty V_n \right)$ is a proper nonzero projection. The spectrum of $N|_{\text{Ran}(P)}$ contains at least $\{x_1, x_2, \ldots \} \cup (\sigma(N) \setminus O)$ by design, and the closure of the union is $\sigma(N)$. This contradicts the minimality of $N$.
\end{proof}

The condition $\overline{\mathcal{I}_K} = \overline{\text{ext}(K)}$ also allows us to expand Corollary \ref{cor:uniqueness_issues} to the non-compact setting (without the need to assume $K$ is a polyhedron). That is, the condition which characterizes existence and uniqueness of minimal normal tuples also guarantees the existence of a plethora of non-normal minimal tuples for the same matrix range $\Wmin{}(K)$. 

\begin{corollary}\label{cor:non-compact_non-unique}
Let $K$ be a compact convex set with at least three extreme points, such that $\overline{\mathcal{I}_K} = \overline{\text{ext}(K)}$. Then there are uncountably many unitarily inequivalent tuples $T$ such that $T$ is minimal for matrix range $\mathcal{W}(T) = \Wmin{}(K)$. For such $T$, the Shilov ideal of $S_T$ in $C^*(S_T)$ is trivial if and only if $T$ is normal.
\end{corollary}
\begin{proof}
Let $\mathcal{I}_K = \{v_1, v_2, v_3, \ldots \}$. Following an affine transformation of Corollary \ref{cor:uniqueness_issues}, choose $T$ to be the direct sum of $v_2, v_3, \ldots$ and an irreducible tuple $S$. The summand $S$ is chosen such that $\mathcal{W}_1(S)$ is contained in the simplex $\text{conv}[v_1, v_2, v_3]$ and $v_1 \in \mathcal{W}_1(S)$, but $\mathcal{W}_1(S)$ does not include any extreme points of $K$ besides $v_1$. The uncountably many options for $T$ are all minimal for matrix range $\mathcal{W}(T) = \Wmin{}(K)$.

Next, let $T$ be any minimal tuple for matrix range $\mathcal{W}(T) = \Wmin{}(K)$, and let $N$ be a diagonal operator with eigenvalues at $\mathcal{I}_K$, which is dense in $\ol{\text{ext}(K)}$. The $C^*$-envelope of $S_T$ is isomorphic to $C(\ol{\text{ext}(K)})$, and in particular is commutative. Therefore, if $T$ is not normal, the Shilov ideal of $S_T$ in $C^*(S_T)$ is nontrivial. However, if $T$ is normal, then by Theorem \ref{thm:easy_spec_thm}, $T$ is unitarily equivalent to $N$. Finally, we verify that the Shilov ideal of $N$ in $S_N$ is trivial, which follows from the fact that $C^*(S_N) \cong C(\ol{\text{ext}(K)})$. In particular, any quotient of $C^*(S_N)$ by a nontrivial ideal $I$ gives rise to a normal tuple $(M_1, \ldots, M_d) = (N_1 + I, \ldots, N_d + I)$ whose joint spectrum $\sigma(M)$ is strictly smaller than $\sigma(N) = \ol{\text{ext}(K)}$. It follows that $W_1(M) = \text{conv}(\sigma(M))$ is a proper subset of $W_1(N) = K$, and the unital map sending $N_i \mapsto M_i$ is certainly not completely isometric. That is, the nontrivial ideal $I$ cannot be the Shilov ideal.
\end{proof}

We also produce another pathological example using the simplex. Consider the universal $C^*$-algebra
\be\label{eq:universal_AC_sec3}
\mathcal{A}_d := C^*(x_1, \ldots, x_d \hspace{4 pt} | \hspace{4 pt} x_i = x_i^*, \hspace{4 pt} x_i^2 = 1, \hspace{4 pt} x_i x_j = -x_j x_i \text{ for } i \not= j),
\ee
which has played a major role, under various guises, in previous problems concerning matrix convex sets and free spectrahedra \cite{originalhelton, DDSS, PSS}. The generators may be realized as a tuple $F^{[d]} = (F_1^{[d]}, \ldots F_d^{[d]})$ of $2^{d-1} \times 2^{d-1}$ matrices defined by $F_1^{[1]} = 1$ and the following recursive identities for $d \geq 2$:
\bes
F_j^{[d]} := F_j^{[d-1]} \otimes \begin{pmatrix} 1 & 0 \\ 0 & -1 \end{pmatrix}, 1 \leq j \leq d - 1,
\ees
\bes
F_{d}^{[d]} := I_{2^{d-2}} \otimes \begin{pmatrix} 0 & 1 \\ 1 & 0 \end{pmatrix}.
\ees

Anticommutation of the $F_i^{[d]}$ implies that for any real $d$-tuple $\lambda$ with $||\lambda||_{\ell^2} = 1$, it follows that $||\lambda_1 F^{[d]}_1 + \ldots + \lambda_d F^{[d]}_d|| \leq 1$. In particular, $\mathcal{W}_1(F^{[d]})$ is contained in the closed $\ell^2$ unit ball $\eucball{d}$, so $F^{[d]} \in \Wmax{}(\eucball{d})$. Elementary computations show that the unit sphere $\mathbb{S}^{d-1}$ is contained in $\mathcal{W}_1(F^{[d]})$, which then implies that $\mathcal{W}_1(F^{[d]}) = \eucball{d}$.

\begin{theorem}\label{thm:body_ball_covering}
Let $K$ be a compact convex set with at least 3 extreme points. Then there is a tuple $A \in \cB(H)^d$ such that $\mathcal{W}(A) = \Wmin{}(K)$, $A$ has no nontrivial normal summands, and any summand $B$ of $A$ such that $\mathcal{W}(B) = \mathcal{W}(A)$ has the property that $B$ is not minimal for its matrix range.
\end{theorem}
\begin{proof}
We may assume that $K$ is a convex body in $\bR^d$, $d \geq 2$, and construct a tuple of self-adjoint operators. For $k \in \bZ^+$, choose $x_k \in \bR^d$ and $c_k \in (0, \infty)$ such that the collection $\{x_1 + c_1 \eucball{d}, x_2 + c_2 \eucball{d}, \ldots \}$ of $\ell^2$-balls has the following properties. For each $k$, there exists a simplex $S_k$ with $x_k + c_k \eucball{d} \subset S_k \subseteq \text{int}(K)$, and for $i \not= j$, the balls $x_i + c_i \eucball{d}$ and $x_j + c_j \cdot \eucball{d}$ do not intersect. Finally, every $\lambda \in \text{ext}(K)$ is a limit point of the union of all the balls.

Let $A^{(k)}$ be the tuple $x_k + c_k \cdot (F^{[d]}_1, \ldots, F^{[d]}_d)$, so that $\mathcal{W}_1(A^{(k)}) = x_k + c_k \eucball{d}$, and let $A = \bigoplus\limits_{k=1}^\infty A^{(k)}$. Then $\mathcal{W}(A)$ is a closed matrix convex set whose first level has $\text{ext}(K) \subseteq \mathcal{W}_1(A)$, which shows that $K \subseteq \mathcal{W}_1(A)$ and $\Wmin{}(K) \subseteq \mathcal{W}(A)$. On the other hand, because $S_k$ is a simplex, we have that each $A^{(k)}$ admits a normal dilation with spectrum inside $S_k$. Therefore, $A$ admits a normal dilation with spectrum inside $K$, and $\mathcal{W}(A) \subseteq \Wmin{}(K)$ holds. Finally, $\mathcal{W}(A) = \Wmin{}(K)$.

Suppose $\cK$ is a nontrivial subspace of $\bigoplus\limits_{i=1}^\infty \bC^{2^{d-1}}$ which is reducing for $A$. Let $P$ denote the projection onto $\cK$ and write $P$ in block form $[P_{ij}]$ with $P_{ij} \in M_{2^{d-1}}$. Fix any $i \not= j$, so that $\mathcal{W}_1(A^{(i)})$ and $\mathcal{W}_1(A^{(j)})$ are disjoint compact convex sets, meaning there is a hyperplane which separates them. We may therefore fix constants $b_1, \ldots, b_d$ such that the self-adjoint operators
\bes
T^{(k)} := \sum_{m=1}^d b_m A^{(k)}_m
\ees 
have $\mathcal{W}_1(T^{(i)})$ and $\mathcal{W}_1(T^{(j)})$ disjoint, which gives that $T^{(i)}$ and $T^{(j)}$ have disjoint spectrum. Letting 
\bes
T = \sum_{m=1}^d b_m A_m,
\ees
we have that $\sigma(T) \subset \bR$ is a compact set which includes the spectrum of any summand $T^{(k)}$. Fix a continuous function $f$ on $\sigma(T)$, which we may apply using the functional calculus, such that $f(T^{(i)}) = 0$ and $f(T^{(j)})$ is invertible. Since $P$ commutes with $T$, we have that $P f(T) = f(T) P$, which in block form shows that $P_{ij} f(T^{(j)}) = f(T^{(i)}) P_{ij}$. By the choice of $f$, we have that $P_{ij} = 0$. Therefore, $P$ is a direct sum of projections $P_{ii} \in M_{2^{d-1}}$ corresponding to reducing subspaces of $A^{(i)}$. 

Since every reducing subspace $\cK$ of $A$ is a direct sum of reducing subspaces for the $A^{(i)}$, it follows that if $A$ has a nontrivial normal summand, then there is some $A^{(i)}$ which has a nontrivial normal summand as well. Since $(F^{[d]}_1, \ldots, F^{[d]}_d)$ has no nontrivial normal summand, this is impossible. Similarly, if $B$ is a summand of $A$ with the same matrix range $\Wmin{}(K)$, then $B$ is a direct sum of $A^{(i)}|_{\cK_i}$ for reducing subspaces $\cK_i$ of $A^{(i)}$. Since each set $\mathcal{W}_1(A^{(i)})$ is contained in the interior of $K$, it follows that $B$ must have infinitely many summands, with detection of the extreme points of $K$ in $\mathcal{W}_1(B)$ unaffected by the removal of one summand. Finally, $B$ is not minimal for its matrix range $\mathcal{W}(B) = \Wmin{}(K)$.
\end{proof}


\section{Scaled Containments}\label{sec:scaled_containments}

In this section, we consider two matrix convex set containments which may be demonstrated by explicit dilation procedures. We first consider the problem of dilating tuples $T \in \cB(H)^d$ of (not necessarily self-adjoint) contractions to normal tuples $N \in \cB(K)^d$ such that $||N_i|| \leq C$ for each $i$. Recall from (\ref{eq:IcantbelieveImissedthisonthefirsttimearound}) that
\be\label{eq:headdesk2}
\prod_{j=1}^d \Wmin{}(\overline{\bD}) = \{T \in \Md: \text{for each } i \in \{1, \ldots, d\}, ||T_i|| \leq 1\}
\ee
and that an abstract dilation result for contractions can be found in Corollary \ref{cor:contractionSDdilation}, in the language of SD-tuples. Below we show that the constant $C = \sqrt{2d}$ can be achieved explicitly.

\begin{lemma}\label{lem:doublesquaretodiamond}
Suppose $S \in \cB(H)^d$ satisfies
\bes
S_1 S_1^* + \ldots + S_d S_d^* = I = S_1^* S_1 + \ldots + S_d^* S_d.
\ees
Then there is a normal dilation $M$ of $S$ with $||M_i|| \leq \sqrt{2d}$ and $M_i M_j = 0$ for $i \not= j$.
\end{lemma}
\begin{proof}
Write $S_j = X_j + i Y_j$, so that the two identities given imply that
\bes
X_1^2 + Y_1^2 + \ldots + X_d^2 + Y_d^2 = I
\ees
and
\bes
[X_1, Y_1] + [X_2, Y_2] + \ldots + [X_d, Y_d] = 0.
\ees
From \cite[Theorem 6.6]{PSS} and Proposition \ref{prop:introcara}, there is a normal dilation $(Q_1, R_1, \ldots, Q_d, R_d)$ of $(X_1, Y_1, \ldots, X_d, Y_d)$ with joint spectrum in the extreme points of $\sqrt{2d} \cdot \Diam_{2d}$. Therefore, the self-adjoint operators $Q_i$ and $R_j$ have norm $\sqrt{2d}$ and satisfy $Q_i R_j = 0$ for all $i, j$ and $Q_i Q_j = 0 = R_i R_j$ if $i \not= j$. It follows that $(S_1, \ldots, S_d)$ has a normal dilation $(M_1, \ldots, M_d) = (Q_1 + iR_1, \ldots, Q_d + iR_d)$ such that $||M_i|| = \sqrt{2d}$ and $M_i M_j = 0$ for $i \not= j$.
\end{proof}

\begin{remark} For $d = 2$, this estimate cannot be improved. Consider the elementary $2 \times 2$ matrices $S_1 = E_{12}$ and $S_2 = E_{21} = S_1^*$, which meet the identities required. If a normal dilation  $(M_1, M_2)$ has $||M_j|| \leq r$ and $M_1 M_2 = 0$, then it follows that $||M_1 + M_2^*|| \leq r$. Since $M_1 + M_2^*$ is a dilation of $S_1 + S_2^* = 2 S_1$, which has norm $2$, this implies that $r \geq 2 = \sqrt{2 d}$.
\end{remark}

\begin{theorem}\label{thm:contraction_exp_dilation}
Suppose $T \in \cB(H)^d$ is a tuple of (not necessarily self-adjoint) contractions. Then there is a normal dilation $N$ of $T$ with $||N_i|| \leq \sqrt{2d}$ for each $i$. It follows that
\be\label{eq:yetanothercontainment}
\prod\limits_{j=1}^d \Wmin{}(\overline{\bD}) \subseteq \sqrt{2d} \cdot \Wmin{}\left(\overline{\mathbb{D}}^d\right).
\ee
The constant $\sqrt{2d}$ is not necessarily optimal.
\end{theorem}
\begin{proof}
For convenience, we label the operators in $T$ as $T_0, \ldots, T_{d-1}$. Use Halmos dilation (if necessary) to obtain a dilation tuple $U = (U_0, \ldots, U_{d-1})$ where each $U_i$ is unitary. Let $\omega$ be a primitive $d$th root of unity, and define the averages
\be\label{eq:av_proc}
S_j := \frac{1}{d} \sum_{k=0}^{d-1} \omega^{jk} U_k.
\ee
A simple computation using the identity $1 + \omega^{r} + \ldots + \omega^{r(d-1)} = 0$ for $\omega^r \not= 1$ shows that
\bes
\sum_{j=0}^{d-1} S_j S_j^* = I = \sum_{j=0}^{d-1} S_j^* S_j.
\ees
By Lemma \ref{lem:doublesquaretodiamond}, there is a normal dilation $M$ of $S$ with $||M_j|| \leq \sqrt{2d}$ and $M_i M_j = 0$ for $i \not=j$. Moreover, it follows from (\ref{eq:av_proc}) that
\bes
N_j := \sum_{n=0}^{d-1} \omega^{-jn} M_n
\ees
is a dilation of $U_j$ (and hence also of $T_j$). Since $M_i M_j = 0$ for $i \not= j$, we have that $||N_j|| \leq \sqrt{2d}$, and $N = (N_0, \ldots, N_{d-1})$ is a normal tuple.
\end{proof}

\begin{corollary}
Suppose $T \in \cK(H)^d$ is a tuple of (not necessarily self-adjoint) compact contractions. Then for any $\varepsilon > 0$, there is a normal tuple $N$ consisting of compact operators such that $N$ is a dilation of $T$ and $||N_i|| \leq \sqrt{2d} + \varepsilon$ for each $i$.
\end{corollary}
\begin{proof}
Apply Proposition \ref{prop:approxcompactdilation} to the claim $\mathcal{W}(T) \subseteq \Wmin{}\left( \sqrt{2d} \cdot \overline{\bD}^d \right)$, which follows from Theorem \ref{thm:contraction_exp_dilation}.
\end{proof}

The constant $\sqrt{2d}$ in Theorem \ref{thm:contraction_exp_dilation} strictly improves the constant $\min\{d, 2 \sqrt{d}\}$ from \cite[Corollary 6.11]{PSS} when $d \geq 3$, but for $d = 2$, the two constants are equal. While Lemma  \ref{lem:doublesquaretodiamond} is optimal when $d = 2$, the proof of Theorem \ref{thm:contraction_exp_dilation} does not use the full strength of the lemma. Namely, the final step of the proof (when $d = 2$) seeks to show that
\bes
N_1 := M_1 + M_2 \hspace{.4 in} \text{and} \hspace{.4 in} N_2 := M_1 - M_2
\ees
are normal operators which commute and have some norm bound $||N_i|| \leq C$. Lemma \ref{lem:doublesquaretodiamond} in the case $d = 2$ shows that $C = \sqrt{2 \cdot 2} = 2$ can be obtained, with the stronger condition that the building blocks $M_1$ and $M_2$ are normal operators with $M_1 M_2 = 0 = M_2 M_1$. Thus, we cannot necessarily conclude that the constant in Theorem \ref{thm:contraction_exp_dilation} is optimal. We remind the reader that due to the estimate in Corollary \ref{cor:contractionSDdilation}, knowledge of the optimal constant in Theorem \ref{thm:contraction_exp_dilation} also produces a bound on $\USD{d}$, which might improve the bound  (\ref{eq:USD_bounds_sqrtandd}).

As noted in (\ref{eq:headdesk2}), the set $\mathcal{S}$ of all $d$-tuples of matrix contractions is equal to $\prod\limits_{j=1}^d \Wmin{}(\overline{\bD})$, and in particular, it is strictly smaller than $\Wmax{}\left(\overline{\mathbb{D}}^d\right)$. In fact, the containment (\ref{eq:yetanothercontainment}) holds even though the larger scale $2 \sqrt{d}$ in 
\bes
\Wmax{}\left(\overline{\mathbb{D}}^d\right) \subseteq 2 \sqrt{d} \cdot \Wmin{}\left(\overline{\mathbb{D}}^d\right)
\ees
is optimal by \cite[Corollary 6.11]{PSS}. However, while $\mathcal{S}$ is not a maximal matrix convex set, $\mathcal{S}$ is trivially equal to the set of matrix tuples $T \in \Md$ such that
\be\label{eq:complex_polar}
\lambda \in \bC^d, \hspace{.1 in} |\lambda_1| + \ldots + |\lambda_d| \leq 1 \hspace{.5 in} \implies \hspace{.5 in} ||\lambda_1 T_1 + \ldots + \lambda_d T_d|| \leq 1.
\ee
The scalar tuple $\lambda$ is selected from the complex $\ell^1$ ball, which is dual to the complex $\ell^\infty$ ball $\overline{\mathbb{D}}^d$. Therefore, (\ref{eq:complex_polar}) may be considered a $\bC$-linear analogue of the real inequalities (\ref{eq:wmaxdef}) that characterize sets of the form $\Wmax{}(K)$. Further, since each $\mathcal{S}_n$ is closed under multiplication by $n \times n$ unitary matrices, $\mathcal{S}$ is a free circular matrix convex set in the sense of \cite[\S 1]{circular}.

We now pursue another matrix convex set containment through explicit dilation. As in (\ref{eq:universal_AC_sec3}), consider the universal $C^*$-algebra
\be\label{eq:universal_AC}
\mathcal{A}_d := C^*(x_1, \ldots, x_d \hspace{4 pt} | \hspace{4 pt} x_i = x_i^*, \hspace{4 pt} x_i^2 = 1, \hspace{4 pt} x_i x_j = -x_j x_i \text{ for } i \not= j)
\ee
and the concrete realization of the generators in the tuple $F^{[d]} = (F_1^{[d]}, \ldots F_d^{[d]})$, where the $F^{[i]}_j$ are $2^{d-1} \times 2^{d-1}$ matrices defined by $F_1^{[1]} = 1$ and the following recursive identities for $d \geq 2$:
\bes
F_j^{[d]} := F_j^{[d-1]} \otimes \begin{pmatrix} 1 & 0 \\ 0 & -1 \end{pmatrix}, 1 \leq j \leq d - 1,
\ees
\bes
F_{d}^{[d]} := I_{2^{d-2}} \otimes \begin{pmatrix} 0 & 1 \\ 1 & 0 \end{pmatrix}.
\ees
If $\eucball{d}$ denotes the closed $\ell^2$ unit ball in $\bR^d$, then \cite[Proposition 14.14]{originalhelton} (adjusted to the self-adjoint complex setting) shows that the free spectrahedron
\bes
D_{F^{[2]}} := \left\{X \in \mathcal{M}^2: X_1 \otimes F^{[2]}_1 + X_2 \otimes F^{[2]}_2 \leq I \right\}
\ees
satisfies
\be\label{eq:spec=min}
D_{F^{[2]}} = \Wmin{}(\eucball{2}).
\ee
Two proofs are given in \cite[Proposition 14.14]{originalhelton}, one of which uses an explicit dilation procedure. The polar dual (see \cite[\S 3]{DDSS} for details) of (\ref{eq:spec=min}) is the equivalent expression 
\be\label{eq:ACrange=max2}
\mathcal{W}(F^{[2]}) = \Wmax{}(\eucball{2}).
\ee
It is not clear how to extend the techniques of \cite[Proposition 14.14]{originalhelton} in order to prove the extension of (\ref{eq:spec=min}) or its dual (\ref{eq:ACrange=max2}) to $d > 2$. We will focus on providing explicit dilation evidence for the dual formulation
\be\label{eq:ACrange=maxd}
\mathcal{W}(F^{[d]}) \queseq \Wmax{}(\eucball{d}),
\ee
noting that the containment $\subseteq$ is trivial. Moreover, while the polar dual may allow one to switch between two equivalent problems, we note that explicit dilation information does not generally survive applying the polar dual.

 Since the members of $F^{[d]}$ are generators of the universal $C^*$-algebra $\mathcal{A}_d$, application of Stinespring factorization shows that a tuple $T \in \Md$ is in $\mathcal{W}(F^{[d]})$ if and only if there is a dilation $A$ of $T$ consisting of self-adjoint unitiaries $A_i$ such that $A_i A_j = -A_j A_i$ for $i \not= j$. Thus, a proof that (\ref{eq:ACrange=maxd}) holds would imply that the existence of such a dilation $A$ is characterized by the satisfaction of linear inequalities by $T$. We will provide some evidence for (\ref{eq:ACrange=maxd}) by using more restrictive linear inequalities, which place $\mathcal{W}_1(T)$ in a rectangular prism inside the ball.

\begin{theorem}\label{thm:cont_to_ACaj}
Let $X = (X_1, \ldots, X_d) \in \cB(H)^d_{sa}$ be a $d$-tuple of self-adjoint contractions, and define $\text{AC}(X) \subseteq \cB(H)_{sa}$ as the collection of all self-adjoint operators which anticommute with each $X_j$. If $a_1, \ldots, a_d > 0$ satisfy $\sum\limits_{j=1}^d a_j^{-2} \leq 1$, then there exists a dilation tuple $A \in \cB(H \otimes \mathbb{C}^{4^{d-1}})^d_{sa}$ of $X$ with the following properties.
\begin{itemize}
\item For each $j$, $||A_j|| \leq a_j$.
\item For $j \not= k$, $A_j$ and $A_k$ anticommute.
\item For each $W \in \text{AC}(X)$, $W \otimes \begin{pmatrix} 1 & 0 \\ 0 & -1 \end{pmatrix}^{\otimes 2(d-1)}$ anticommutes with every $A_j$.
\item The block entries of each $A_j$ are in the real unital $C^*$-algebra generated by $X_1, \ldots, X_d$.
\end{itemize}
\end{theorem}
\begin{proof}
The case $d = 1$ is trivial, as no dilation is necessary. We proceed by induction: suppose the theorem holds for $d - 1$. Given a tuple $X \in \cB(H)^d_{sa}$ of self-adjoint contractions, consider first the Halmos dilations $Y_i = \begin{pmatrix} X_i & \sqrt{1-X_i^2} \\ \sqrt{1-X_i^2} & -X_i \end{pmatrix} \in \cB(H \otimes \mathbb{C}^2)_{sa}$, which anticommute with $\begin{pmatrix} W & 0 \\ 0 & -W \end{pmatrix} = W \otimes \begin{pmatrix} 1 & 0 \\ 0 & -1 \end{pmatrix}$ for each $W \in \text{AC}(X)$. Given $a_1, \ldots, a_d > 0$ with $\sum\limits_{j=1}^d a_j^{-2} \leq 1$, let $a_d = \sqrt{1+r^2}$, and define $b_j = \frac{a_j}{\sqrt{1+1/r^2}}$ for $1 \leq j \leq d - 1$, so that $\sum\limits_{j=1}^{d-1} b_j^{-2} \leq 1$. Make intermediate dilations as follows:

\[
G_j := \begin{pmatrix} Y_j & \cfrac{Y_j Y_d + Y_d Y_j}{2 r} \\ \cfrac{Y_j Y_d + Y_d Y_j}{2 r} & -Y_j \end{pmatrix}, \hspace{.06 in} 1 \leq j \leq d-1,\hspace{.5 in}
E := \begin{pmatrix} Y_d & -r I \\ -r I & -Y_d \end{pmatrix},
\]
so that $E$ anticommutes with $G_1, \ldots, G_{d-1}$. Within each operator, the diagonal term anticommutes with the off-diagonal term, so applying the $C^*$-norm identity $||A|| = \sqrt{||AA^*||}$ to these self-adjoint operators shows that
\[
||G_j|| \leq \sqrt{1^2 + \frac{1}{r^2}} = \frac{a_j}{b_j}, \hspace{.06 in} 1 \leq j \leq d - 1, \hspace{.65 in}
||E|| \leq \sqrt{1^2 + r^2} = a_d.
\]
Moreover, since $Y_1, \ldots, Y_d$ anticommute with $W \otimes \begin{pmatrix} 1 & 0 \\ 0 & -1 \end{pmatrix}$ for $W \in \text{AC}(X)$, it follows that $G_1, \ldots, G_{d-1}$ anticommute with $W \otimes \begin{pmatrix} 1 & 0 \\ 0 & -1 \end{pmatrix}^{\otimes 2}$. That is,
\be\label{eq:inductive_AC_stuff}
\{E\} \cup \left\{W \otimes \begin{pmatrix} 1 & 0 \\ 0 & -1 \end{pmatrix}^{\otimes 2}: W \in \text{AC}(X) \right\} \subseteq \text{AC}(G).
\ee

Apply the inductive assumption (scaled by $\frac{a_j}{b_j}$) to the tuple $G \in \cB(H \otimes \mathbb{C}^4)^{d-1}_{sa}$, the collection $\text{AC}(G)$, and the scalars $b_1, \ldots, b_{d-1}$. It follows that there exist pairwise anticommuting dilations $A_1, \ldots, A_{d-1} \in \cB(H \otimes \mathbb{C}^4 \otimes \mathbb{C}^{4^{d-2}})_{sa} = \cB(H \otimes \bC^{4^{d-1}})_{sa}$ of $G_1, \ldots, G_{d-1}$ with norm $||A_j|| \leq \frac{a_j}{b_j} \cdot b_j = a_j$ and with block entries in the real unital $C^*$-algebra generated by $G_1, \ldots, G_{d-1}$. Moreover, for each $V \in \text{AC}(G)$, $A_1, \ldots, A_{d-1}$ anticommute with $V \otimes \begin{pmatrix} 1 & 0 \\ 0 & -1 \end{pmatrix}^{\otimes 2(d-2)}$. By (\ref{eq:inductive_AC_stuff}), for each $W \in \text{AC}(X)$, $A_1, \ldots, A_{d-1}$ anticommute with $W \otimes \begin{pmatrix} 1 & 0 \\ 0 & -1 \end{pmatrix}^{\otimes 2} \otimes \begin{pmatrix} 1 & 0 \\ 0 & -1 \end{pmatrix}^{\otimes 2(d-2)} = W  \otimes \begin{pmatrix} 1 & 0 \\ 0 & -1 \end{pmatrix}^{\otimes 2(d-1)}$. Similarly, if $A_d := E \otimes \begin{pmatrix} 1 & 0 \\ 0 & -1 \end{pmatrix}^{\otimes 2(d-2)}$, then $||A_d|| =||E|| \leq a_j$, and (\ref{eq:inductive_AC_stuff}) shows that $A_d$ anticommutes with $A_1, \ldots, A_{d-1}$. Finally, since $E$ anticommutes with $W \otimes \begin{pmatrix} 1 & 0 \\ 0 & -1 \end{pmatrix}^{\otimes 2}$ for each $W \in \text{AC}(X)$, it follows that $A_d$ anticommutes with $W \otimes \begin{pmatrix} 1 & 0 \\ 0 & -1 \end{pmatrix}^{\otimes 2(d-1)}$. The inductive step is complete, as an examination of the intermediate dilations shows that the block entries of $A_d$ belong to the real unital $C^*$-algebra generated by the $X_i$.
\end{proof}

We note that the operators $A_j$ constructed in Theorem \ref{thm:cont_to_ACaj} have $||A_j|| = a_j$, but they do not necessarily satisfy $A_j^2 = a_j^2 I$. This may be remedied by a \lq\lq step-by-step\rq\rq\hspace{0pt} modification of the Halmos dilation procedure, designed to preserve pairwise anticommutation. We include the details for completeness.

\begin{proposition}\label{prop:stepbystep_daybyday}
Let $A \in \cB(H)^d_{sa}$ be a tuple of pairwise anticommuting self-adjoint operators with $||A_j|| \leq a_j$. Then there is a dilation $M \in \cB(H \otimes \bC^{2^{d}})^d_{sa}$ consisting of pairwise anticommuting self-adjoint operators $M_j$ with $M_j^2 = a_j^2 I$. Moreover, we may choose $M_j$ such that for every $X \in \cB(H)_{sa}$ that anticommutes with $A_1, \ldots, A_d$, $\bigoplus\limits_{i=1}^{2^d} X$ anticommutes with $M_1, \ldots, M_d$.
\end{proposition}
\begin{proof}
The case $d = 1$ follows from the Halmos dilation $M_ 1 = \begin{pmatrix} A_1 & \sqrt{a_1^2 I - A_1^2} \\ \sqrt{a_1^2 I - A_1^2} & -A_1 \end{pmatrix}$, as if $X$ anticommutes with $A_1$, then $X \oplus X$ anticommutes with $M_1$. Therefore, we may induct, so we assume the result holds for $d - 1$. If $A \in \cB(H)^d_{sa}$ is a tuple of pairwise anticommuting self-adjoint operators with $||A_j|| \leq a_j$, and we let $B = (A_1, \ldots, A_{d-1})$, then $B$ admits a dilation $L \in \cB(H \otimes \bC^{2^{d-1}})_{sa}$ with $L_j^2 = a_j^2 I$ and $L_j L_k = - L_k L_j$ for $j \not= k$. Moreover, we may choose $L_j$ such that if $X$ anticommutes with $A_1, \ldots, A_{d-1}$, then $\bigoplus\limits_{i=1}^{2^{d-1}} X$ anticommutes with $L_1, \ldots, L_{d-1}$. In particular, this applies to $A_d$, and we may define
\bes
M_j := \begin{pmatrix} L_j & 0 \\ 0 & -L_j \end{pmatrix}, \hspace{.05 in} 1 \leq j \leq d - 1, \hspace{.28 in} M_{d+1} := \begin{pmatrix} \bigoplus\limits_{i=1}^{2^{d-1}} A_{d} & \sqrt{a_{d}^2 I - \bigoplus\limits_{i=1}^{2^{d-1}} A_{d}^2} \\ \sqrt{a_{d}^2 I - \bigoplus\limits_{i=1}^{2^{d-1}} A_{d}^2} & -\bigoplus\limits_{i=1}^{2^{d-1}} A_{d} \end{pmatrix}.
\ees
It follows that $M_j^2 = a_j^2 I$ for all $j$, $M_j M_k = - M_k M_j$ if $j \not= k$, and if $X$ anticommutes with $A_1, \ldots, A_d$, then $\bigoplus\limits_{j=1}^{2^d} X$ anticommutes with $M_1, \ldots, M_d$.
\end{proof}

Finally, we may place a lower bound on the matrix range of the tuple $F^{[d]}$, which consists of (universal) pairwise anticommuting, self-adjoint unitaries.

\begin{corollary}\label{cor:cubeball}
If $c_1, \ldots, c_d \geq 0$ satisfy $\sum c_j^2 \leq 1$, then $\Wmax{}\left( \prod\limits_{j=1}^d [-c_j, c_j] \right) \subseteq \mathcal{W}(F^{[d]})$.
\end{corollary}
\begin{proof}
We may assume $c_j > 0$, as $\mathcal{W}(F^{[d]})$ is closed. Given a tuple $X \in \Wmax{}\left( \prod\limits_{j=1}^d [-c_j, c_j] \right)$, it follows that $B := (\frac{1}{c_1} X_1, \ldots, \frac{1}{c_d} X_d)$ is a tuple of self-adjoint contractions. Since $\sum\limits_{j=1}^d \left( \frac{1}{c_j}\right)^{-2} = \sum\limits_{j=1}^d c_j^2 \leq 1$, Theorem \ref{thm:cont_to_ACaj} and Proposition \ref{prop:stepbystep_daybyday} shows that $B$ admits a dilation $M$ consisting of pairwise anticommuting self-adjoint operators with $M_j^2 = \frac{1}{c_j^2} I$ for each $j$. Rescaling by $c_j$ shows that $X$ admits a dilation $L$ consisting of pairwise anticommuting self-adjoint unitaries. This implies that $L_1, \ldots, L_d$ satisfy the relations of (\ref{eq:universal_AC}), and hence there is a unital $*$-homomorphism $F^{[d]}_j \mapsto L_j$. Composition with a compression shows that there is a UCP map $F^{[d]}_j \mapsto X_j$, and finally $X \in \mathcal{W}(F^{[d]})$.
\end{proof}

With (\ref{eq:ACrange=max2}) and Corollary \ref{cor:cubeball}, it is within the realm of possibility that (\ref{eq:ACrange=maxd}) holds in full generality. We conclude by noting that
\bes
\mathcal{A}_d := C^*(x_1, \ldots, x_d \hspace{4 pt} | \hspace{4 pt} x_i = x_i^*, \hspace{4 pt} x_i^2 = 1, \hspace{4 pt} x_i x_j = -x_j x_i \text{ for } i \not= j)
\ees
is but one of many universal $C^*$-algebras that produces potential scaled containments
\bes
\frac{1}{C} \cdot \Wmax{}(K) \subseteq \mathcal{W}( x_1, \ldots, x_d)
\ees
that may be posed as (noncommutative) dilation problems through Stinespring factorization. It would be of great interest to the author if, in addition to a resolution of (\ref{eq:ACrange=maxd}), there were a general method by which one could compute or bound the optimal scale $C$, depending on the universal $C^*$-algebra and the relationship between $K$ and $\mathcal{W}_1(x_1, \ldots, x_d)$.


\section*{Acknowledgments}

\noindent I am grateful to Orr Shalit and Adam Dor-On for their comments, and to the referee for significant improvements  (especially regarding section \ref{sec:compmin} and alternative notions of minimality).



\begin{thebibliography}{10}

\bibitem{arvesonsubI}
William Arveson.
\newblock Subalgebras of {$C^{\ast} $}-algebras.
\newblock {\em Acta Math.}, 123:141--224, 1969.

\bibitem{arvesonsubII}
William Arveson.
\newblock Subalgebras of {$C^{\ast} $}-algebras. {II}.
\newblock {\em Acta Math.}, 128(3-4):271--308, 1972.

\bibitem{arvesonmatrix}
William Arveson.
\newblock The noncommutative {C}hoquet boundary {III}: operator systems in
  matrix algebras.
\newblock {\em Math. Scand.}, 106(2):196--210, 2010.

\bibitem{arvesonhyp}
William Arveson.
\newblock The noncommutative {C}hoquet boundary {II}: hyperrigidity.
\newblock {\em Israel J. Math.}, 184:349--385, 2011.

\bibitem{choiUCP}
Man~Duen Choi.
\newblock Completely positive linear maps on complex matrices.
\newblock {\em Linear Algebra and Appl.}, 10:285--290, 1975.

\bibitem{choieffros}
Man~Duen Choi and Edward~G. Effros.
\newblock Injectivity and operator spaces.
\newblock {\em J. Functional Analysis}, 24(2):156--209, 1977.

\bibitem{DDSScorrected}
Kenneth~R. Davidson, Adam Dor-On, Orr~Moshe Shalit, and Baruch Solel.
\newblock {D}ilations, inclusions of matrix convex sets, and completely
  positive maps. {C}orrected version in arxiv:1601.07993v3.

\bibitem{DDSS}
Kenneth~R. Davidson, Adam Dor-On, Orr~Moshe Shalit, and Baruch Solel.
\newblock Dilations, inclusions of matrix convex sets, and completely positive
  maps.
\newblock {\em Int. Math. Res. Not. IMRN}, (13):4069--4130, 2017.

\bibitem{effroswinkler}
Edward~G. Effros and Soren Winkler.
\newblock Matrix convexity: operator analogues of the bipolar and
  {H}ahn-{B}anach theorems.
\newblock {\em J. Funct. Anal.}, 144(1):117--152, 1997.

\bibitem{evertabs}
Eric Evert.
\newblock Matrix convex sets without absolute extreme points.
\newblock {\em Linear Algebra Appl.}, 537:287--301, 2018.

\bibitem{evertheltonext}
Eric Evert, J.~William Helton, Igor Klep, and Scott McCullough.
\newblock Extreme points of matrix convex sets, free spectrahedra and dilation
  theory. arxiv:1612.00025.

\bibitem{circular}
Eric Evert, J.~William Helton, Igor Klep, and Scott McCullough.
\newblock Circular free spectrahedra.
\newblock {\em J. Math. Anal. Appl.}, 445(1):1047--1070, 2017.

\bibitem{FNT}
Tobias Fritz, Tim Netzer, and Andreas Thom.
\newblock Spectrahedral containment and operator systems with
  finite-dimensional realization.
\newblock {\em SIAM J. Appl. Algebra Geom.}, 1(1):556--574, 2017.

\bibitem{halmos}
Paul~R. Halmos.
\newblock Normal dilations and extensions of operators.
\newblock {\em Summa Brasil. Math.}, 2:125--134, 1950.

\bibitem{hamana}
Masamichi Hamana.
\newblock Injective envelopes of {$C^{\ast} $}-algebras.
\newblock {\em J. Math. Soc. Japan}, 31(1):181--197, 1979.

\bibitem{relaxation}
J.~William Helton, Igor Klep, and Scott McCullough.
\newblock The matricial relaxation of a linear matrix inequality.
\newblock {\em Math. Program.}, 138(1-2, Ser. A):401--445, 2013.

\bibitem{originalhelton}
J.~William Helton, Igor Klep, Scott McCullough, and Markus Schweighofer.
\newblock Dilations, linear matrix inequalities, the matrix cube problem and
  beta distributions. arxiv:1412.1481.

\bibitem{evgenios}
Evgenios~T.A. Kakariadis.
\newblock Notes on the ${C}^*$-envelope and the \u{S}ilov ideal.

\bibitem{kreinmilman}
M.~Krein and D.~Milman.
\newblock On extreme points of regular convex sets.
\newblock {\em Studia Math.}, 9:133--138, 1940.

\bibitem{Kriel}
Tom-Lukas Kriel.
\newblock An introduction to matrix convex sets and free spectrahedra.
  arxiv:1611.03103.

\bibitem{PSS}
Benjamin Passer, Orr~Moshe Shalit, and Baruch Solel.
\newblock Minimal and maximal matrix convex sets.
\newblock {\em J. Funct. Anal.}, 274(11):3197--3253, 2018.

\bibitem{steinitz}
E.~Steinitz.
\newblock Bedingt konvergente {R}eihen und konvexe {S}ysteme.
\newblock {\em J. Reine Angew. Math.}, 146:1--52, 1916.

\bibitem{stinespring}
W.~Forrest Stinespring.
\newblock Positive functions on {$C^*$}-algebras.
\newblock {\em Proc. Amer. Math. Soc.}, 6:211--216, 1955.

\bibitem{zalar}
Alja\v{z} Zalar.
\newblock Operator {P}ositivstellens\"atze for noncommutative polynomials
  positive on matrix convex sets.
\newblock {\em J. Math. Anal. Appl.}, 445(1):32--80, 2017.

\end{thebibliography}
\end{document}